\documentclass[10.5pt, a4paper]{amsart}
\usepackage{amssymb, amsmath, amscd, bm, amsthm, pdfpages, setspace, hyperref, mathtools, subcaption, graphicx, cite}

\def\@Rref#1{\hbox{\rm \ref{#1}}}
\def\Rref#1{\@Rref{#1}}
\theoremstyle{plain}
\newtheorem{theorem}{Theorem}[section]
\newtheorem{proposition}[theorem]{Proposition}

\newtheorem{lemma}[theorem]{Lemma}

\theoremstyle{definition}
\newtheorem{definition}{Definition}[section]
\newtheorem{example}[definition]{Example}
\newtheorem{remark}[definition]{Remark}

\newcommand{\im}{\mathop{\rm Im}\nolimits}
\newcommand{\re}{\mathop{\rm Re}\nolimits}

\begin{document}

\title[Cayley transforms and inverses of semigroup generators]{Decay estimates for
	Cayley transforms and inverses of semigroup generators
	via the $\mathcal{B}$-calculus}

\thispagestyle{plain}

\author{Masashi Wakaiki}
\address{Graduate School of System Informatics, Kobe University, Nada, Kobe, Hyogo 657-8501, Japan}
 \email{wakaiki@ruby.kobe-u.ac.jp}
 \thanks{This work was supported by JSPS KAKENHI Grant Number JP20K14362.}

\begin{abstract}
Let $-A$ be the generator of a bounded $C_0$-semigroup $(e^{-tA})_{t \geq 0}$ 
on a Hilbert space.
First we study the long-time  asymptotic behavior of the Cayley transform
$V_{\omega}(A) \coloneqq (A-\omega I) (A+\omega I)^{-1}$ with $\omega >0$.
We give a decay estimate for
$\|V_{\omega}(A)^nA^{-1}\|$ when $(e^{-tA})_{t \geq 0}$ is polynomially stable.
Considering the case where the parameter $\omega$ varies, we estimate
$\|(\prod_{k=1}^n V_{\omega_k}(A))A^{-1}\|$ for exponentially stable  $C_0$-semigroups
$(e^{-tA})_{t \geq 0}$.
Next 
we show that if the generator $-A$ of the bounded $C_0$-semigroup
has a bounded inverse, then 
$\sup_{t \geq 0} \|e^{-tA^{-1}} A^{-\alpha} \| < \infty$
for all $\alpha >0$.
We also present an estimate for the rate of decay of $\|e^{-tA^{-1}} A^{-1} \|$,
assuming that $(e^{-tA})_{t \geq 0}$ is polynomially stable.
To obtain these results, we use operator norm estimates offered by
a functional calculus called the $\mathcal{B}$-calculus.
\end{abstract}



\maketitle

\section{Introduction}
\label{sec:introduction}
When solving differential equations numerically, we employ
time-discretization methods.
The first topic is the long-time asymptotic behavior of numerical solutions.
Let $-A$ be a bounded $C_0$-semigroup $(e^{-tA})_{t \geq 0}$ on a Banach space $X$, and
consider the differential equation
\begin{equation}
	\label{eq:differential_eq}
	\dot x = -A x,\quad t \geq 0;\qquad x(0) = x_0 \in X.
\end{equation}
The Crank--Nicolson discretization scheme with parameter $\omega >0$ 
transforms the differential equation \eqref{eq:differential_eq}
into the following difference equation:
\begin{equation}
	\label{eq:difference_eq}
	x_d(n+1) = -V_{\omega}(A) x_d(n),\quad n \in \mathbb{N}_0 \coloneqq \{0,1,2,\dots \};\quad x_d(0) = x_0 \in X,
\end{equation}
where $V_{\omega}(A)$ is defined by $V_{\omega}(A) \coloneqq (A - \omega I)(A+ \omega I)^{-1}$ and is called {\em the Cayley transform of $A$ with parameter $\omega$}.
Assume that $A$ has a bounded inverse. To 
obtain a uniform rate of decay of
numerical  solutions $x_d$ starting in $D(A)$, we investigate
the quantitative behavior of the operator norm 
$\|V_{\omega}(A)^nA^{-1}\|$
as $n \to \infty$.

For simplicity of notation, let $V(A) \coloneqq V_{1}(A)$.
The asymptotic behavior of $(V(A)^n )_{n \in \mathbb{N}_0}$ has been extensively studied;
see the survey \cite{Gomilko2017} and the book \cite[Chapter~5]{Eisner2010}.
Some particular relevant studies will be cited below.
For the generator $-A$ of a
bounded $C_0$-semigroup on a Banach space, $\|V(A)^n\| = O(\sqrt{n})$
as $n \to \infty$ holds, i.e.,
there exist constants $M>0$ and $n_0 \in \mathbb{N}$ 
such that $\|V(A)^n\| \leq M\sqrt{n}$ for all $n \geq n_0$, 
and this estimate cannot be  improved in general; see \cite{Brenner1979, Piskarev2007, Esiner2008JEE, Gomilko2011}.
In the Hilbert space setting, it is still unknown whether the boundedness 
of $(e^{-tA})_{t \geq 0}$ 
implies that of $(V(A)^n )_{n \in \mathbb{N}_0}$.
For bounded $C_0$-semigroups on Hilbert spaces, the estimates
\begin{equation}
	\label{eq:CT_bound1}
	\|V(A)^n\| = O(\log n)\qquad (n \to \infty)
\end{equation}
and 
\begin{equation}
	\label{eq:CT_bound2}
	\sup_{n \in \mathbb{N}_0}\|V(A)^n
	(I+A)^{-\alpha}\| < \infty,\quad \alpha > 0
\end{equation}
remain the best so far.
The former estimate \eqref{eq:CT_bound1} has been obtained in \cite{Gomilko2004}, and
one can derive the latter estimate \eqref{eq:CT_bound2} by combining
\cite[Lemma~2.3]{Gomilko2023}
and the theory of the
$\mathcal{B}$-calculus established in \cite{Batty2021,Batty2021JFA}.
If both $A$ and $A^{-1}$ generate a bounded $C_0$-semigroup on a Hilbert space, 
then
$(V(A)^n)_{n \in \mathbb{N}_0}$ 
is bounded, which has been independently proved in \cite{Azizov2004,Gomilko2004,
	Guo2006}.
Under the additional assumption that $(e^{-tA})_{t \geq 0}$ is strongly stable,
$(V(A)^n)_{n \in \mathbb{N}_0}$ is also strongly stable; see \cite{Guo2006}.
The approach developed in \cite{Guo2006} is based on Lyapunov equations and
has been extended in 
\cite{Piskarev2007} from 
the case of constant parameters $V(A)^n$
to the case of variable parameters $\prod_{k=1}^n V_{\omega_k}(A)$, where 
the sequence $(\omega_k)_{k \in \mathbb{N}_0}$
satisfies $0 < \inf_{k \in \mathbb{N}}\omega_k$ and  $\sup_{k \in \mathbb{N}}\omega_k< \infty$.

In this paper, we consider 
bounded $C_0$-semigroups on Hilbert
spaces with certain decay properties.
In particular, we focus on polynomial stability and exponential stability, which are 
defined as follows:
\begin{definition}
	A $C_0$-semigroup $(e^{-tA})_{t \geq 0}$  on a Banach space is called
	\begin{enumerate}
		\renewcommand{\labelenumi}{(\alph{enumi})}
		\item 
		{\em polynomially stable with parameter $\beta>0$}
		if  $(e^{-tA})_{t \geq 0}$ is bounded and satisfies $\|e^{-tA}(I+A)^{-1}\| = O(t^{-1/\beta})$
		as $t \to \infty$, or simply {\em polynomially stable} if it is polynomially stable with some parameter; and
		\item {\em exponentially stable}
		if $\lim_{t \to \infty} \|e^{\varepsilon t}e^{-tA}\| = 0$ for some $\varepsilon >0$. 
	\end{enumerate}
\end{definition}
Note that if 
the $C_0$-semigroup $(e^{-tA})_{t \geq 0}$ is polynomially stable or exponentially stable, then
the negative generator $A$ has a bounded inverse; see \cite[Theorem~1.1]{Batty2008}.
For more information on these stability notions, 
we refer to the survey article \cite{Chill2020}.

Several results on the rate of decay of $\|V(A)^nA^{-1}\|$ 
have been developed in the Hilbert space setting.
If $(V(A)^n)_{n \in \mathbb{N}_0}$ is bounded, then
\begin{equation}
	\label{eq:CT_poly_pre}
	\|V(A)^nA^{-1}\| = O\left( 
	\left(
	\frac{\log n}{n}
	\right)^{1/(2+\beta)}
	\right) \qquad (n \to \infty)
\end{equation}
for all polynomially stable $C_0$-semigroups with parameter $\beta>0$; see
\cite{Wakaiki2021JEE}.
It has been also proved in \cite{Wakaiki2021JEE} 
that if $-A$ is normal and generates a polynomially stable $C_0$-semigroup with parameter $\beta>0$, then $\|V(A)^nA^{-1}\| = O(n^{-1/(2+\beta)})$
as $n \to \infty$ and the decay rate $n^{-1/(2+\beta)}$
cannot be replaced by a better one in general.
In the case where
$(e^{-tA})_{t \geq 0}$ is exponentially stable and $(e^{-tA^{-1}})_{t \geq 0}$
is bounded, 
the following estimate has been derived in \cite{Wakaiki2023IEOT}:
\begin{equation}
	\label{eq:CT_ex_pre}
	\left\|\left(\prod_{k=1}^n V_{\omega_k}(A)\right)A^{-1}\right\| =
	O\left( 
	\sqrt{\frac{\log n}{n}}
	\right)
	\qquad (n \to \infty),
\end{equation}
where $(\omega_k)_{k \in \mathbb{N}}$
satisfies
$0 < \inf_{k \in \mathbb{N}}\omega_k$ and  $\sup_{k \in \mathbb{N}}\omega_k< \infty$.
In the proof of the estimates \eqref{eq:CT_poly_pre} and \eqref{eq:CT_ex_pre},
Lyapunov equations on $V(A)$ and $-A^{-1}$
play an important role, as in \cite{Guo2006,Piskarev2007}. Hence 
the assumptions on the boundedness of $(V(A)^n)_{n \in \mathbb{N}_0}$ and $(e^{-tA^{-1}})_{t \geq 0}$ are placed.

Our aim is to estimate the decay rate of $\|V(A)^nA^{-1}\|$ without assuming
that $(V(A)^n)_{n \in \mathbb{N}_0}$ or $(e^{-tA^{-1}})_{t \geq 0}$ is bounded.
To this end, we apply the $\mathcal{B}$-calculus.
First, we derive a new estimate for operator norms in the context of 
the $\mathcal{B}$-calculus, which is tailored to 
polynomially stable $C_0$-semigroups on Hilbert spaces.
Using this estimate, we next show that 
\begin{equation}
	\label{eq:CT_poly_intro}
	\|V(A)^nA^{-1}\| = O\left( 
	\frac{\log n}{n^{1/(2+\beta)}}
	\right) \qquad (n \to \infty)
\end{equation}
for all polynomially stable $C_0$-semigroups $(e^{-tA})_{t \geq 0}$ on Hilbert spaces.
The estimate \eqref{eq:CT_poly_intro} 
is less sharp than \eqref{eq:CT_poly_pre} because it has
an additional logarithmic term, and we do not know whether the logarithmic term can be 
omitted.
However, it should be emphasized that 
the boundedness of $(V(A)^n)_{n \in \mathbb{N}_0}$ is not assumed.
When $-A$ is the generator of  an 
exponentially stable $C_0$-semigroup on a Hilbert space, we obtain
the estimate 
\begin{equation}
	\label{eq:CT_ex_intro}
	\left\|\left(\prod_{k=1}^n V_{\omega_k}(A)\right)A^{-1}\right\| =
	O\left( 
	\frac{1}{\sqrt{n}}
	\right)
	\qquad (n \to \infty),
\end{equation}
where $(\omega_k)_{k \in \mathbb{N}}$
satisfies 
$0 < \inf_{k \in \mathbb{N}}\omega_k$ and  $\sup_{k \in \mathbb{N}}\omega_k< \infty$.
This result implies that not only the assumption
on the boundedness of $(e^{-tA^{-1}})_{t \geq 0}$ but also
the logarithmic term $\sqrt{\log n}$ in the previous result \eqref{eq:CT_ex_pre} 
can be omitted. 
Moreover, we show that
the estimate \eqref{eq:CT_ex_intro} cannot be  improved
in the case $A \coloneqq iB+I$, where $B$ is a self-adjoint operator
with a certain spectral property.
A similar improved estimate is provided for a polynomially stable
$C_0$-semigroup on a Hilbert space such that the generator is normal.

We now turn our attention to the long-time
asymptotic  behavior of $(e^{-t A^{-1}})_{t \geq 0}$.
\sloppy
Let
$(e^{-tA})_{t \geq 0}$ be 
a bounded $C_0$-semigroup on a Banach space, and assume that
there exists a densely-defined algebraic inverse $A^{-1}$ of $A$.
The inverse generator problem raised by deLaubenfels \cite{deLaubenfels1988} asks whether
$-A^{-1}$ also generates a $C_0$-semigroup, and a survey can be found in
\cite{Gomilko2017}.
For bounded $C_0$-semigroups on Banach spaces, 
a negative answer to this  problem has been provided implicitly in \cite[pp.~343--344]{Komatsu1966}, and
other counterexamples can be found in \cite{Zwart2007, Gomilko2007MS,Fackler2016}.
The inverse generator problem in the Hilbert space setting remains open.
It has been shown in \cite{Zwart2007} that 
if the answer is positive for all generators and all Hilbert spaces,
then the $C_0$-semigroup $(e^{-tA^{-1}})_{t \geq 0}$ is bounded.
For exponentially stable $C_0$-semigroups on Hilbert spaces,
the estimate $\|e^{-tA^{-1}}\| = O(\log t)$ as $t \to \infty$ has been
obtained by using Lyapunov equations in \cite{Zwart2007SF} and by connecting
the growth bounds of $(e^{-tA^{-1}})_{t \geq 0}$ and 
$(V(A)^n )_{n \in \mathbb{N}_0}$ in \cite{Gomilko2011}.
This estimate has been extended to
bounded $C_0$-semigroups on Hilbert spaces with 
invertible generators by means of the $\mathcal{B}$-calculus in
\cite[Corollary~5.7]{Batty2021}.

Since
we are interested in the asymptotics of $e^{-tA^{-1}} x$ that is uniform with respect to 
$x \in D(A)$,
the problem we study is to estimate $\|e^{-tA^{-1}}A^{-1}\|$.
If $(e^{-t A^{-1}})_{t \geq 0}$ is an exponentially stable $C_0$-semigroup on a Banach space,
then 
$\|e^{-tA^{-1}} A^{-k}\| = O(t^{-k/2+1/4})$ for all $k \in \mathbb{N}$, which has been
established for the case $k=1$ in \cite{Zwart2007} 
and for the case $k \geq 2 $ in \cite{deLaubenfels2009}.
The $\mathcal{B}$-calculus has been applied in \cite{Wakaiki2023IEOT}
in order to
obtain the estimate 
\begin{equation}
	\label{eq:inv_gen_ex_intro}
	\|e^{-tA^{-1}} A^{-\alpha}\| = O\left( \frac{1}{t^{\alpha/2}} \right)\qquad (n \to \infty)
\end{equation} 
with
$\alpha>0$ for all
exponentially stable $C_0$-semigroups on Hilbert spaces.
If the generator $-A$ of a polynomially stable 
$C_0$-semigroup with parameter $\beta>0$
on a Hilbert space is normal,
then 
\begin{equation}
	\label{eq:inv_gen_poly_intro}
	\|e^{-tA^{-1}} A^{-\alpha}\| = O\left( \frac{1}{t^{\alpha/(2+\beta)}} \right)\qquad (n \to \infty)
\end{equation} 
for all $\alpha>0$; see \cite{Wakaiki2023IEOT}. In \cite{Wakaiki2023IEOT},
simple examples have been also provided to show that
the estimates \eqref{eq:inv_gen_ex_intro} and \eqref{eq:inv_gen_poly_intro} cannot be
improved in general.
For polynomially stable $C_0$-semigroups
with any parameters on Hilbert spaces,
the estimate 
\begin{equation}
	\label{eq:inv_gen_poly_bounded_intro}
	\sup_{t \geq 0}\|e^{-tA^{-1}} A^{-1}\| < \infty
\end{equation} 
has been derived in \cite{Wakaiki2023IEOT} by 
the Lyapunov-based approach.

Our second contribution is to improve the estimate \eqref{eq:inv_gen_poly_bounded_intro}
in two ways by means of the $\mathcal{B}$-calculus.
First, we show that $\sup_{t \geq 0}\|e^{-tA^{-1}} A^{-\alpha}\| < \infty$ 
for all $\alpha >0$
when $-A$ is the generator of a
bounded $C_0$-semigroup on a Hilbert space and has a bounded inverse.
Second, we give the estimate of the decay rate
\begin{equation}
	\label{eq:Inv_poly_intro}
	\|e^{-tA^{-1}} A^{-1}\| = O\left( 
	\frac{\log t}{t^{1/(2+\beta)}}
	\right) \qquad (n \to \infty)
\end{equation}
for all polynomially stable $C_0$-semigroups with parameter $\beta>0$ on 
Hilbert spaces. For this, 
the operator norm estimate exploiting polynomial stability is again applied 
in the setting of the $\mathcal{B}$-calculus.
We see  from the estimate \eqref{eq:Inv_poly_intro} that a decay estimate for
$(e^{-tA})_{t \geq 0}$ is transferred to that for
$(e^{-tA^{-1}})_{t \geq 0}$ as in the estimate \eqref{eq:CT_poly_intro} 
for $(V(A)^n)_{n \in \mathbb{N}_0}$.
There is a question whether the logarithmic term in
the estimate \eqref{eq:Inv_poly_intro} can be omitted, and it
is open as in other estimates.

This paper is organized as follows:
In Section~\ref{sec:B_calculus}, we present the notion 
and some properties of the $\mathcal{B}$-calculus.
Section~\ref{sec:CT} is devoted to estimating the rate of decay of
the Cayley transform.
In Section~\ref{sec:IG}, we study the asymptotic behavior of
$\|e^{-tA^{-1}} A^{-\alpha}\|$ with $\alpha >0$.

\subsection*{Notation}
Let $\mathbb{C}_+ \coloneqq \{ z \in \mathbb{C}: \re z > 0 \}$
and let $\mathbb{R}_+ \coloneqq \{ \xi \in \mathbb{R}: \xi \geq 0 
\}$.
Let $\mathbb{Z}$, $\mathbb{N}$, and $\mathbb{N}_0$
denote the set of integers, 
the set of positive integers, and 
the set of nonnegative integers,
respectively.
For a real number $\xi$, let $\lfloor \xi \rfloor \coloneqq
\max\{k \in \mathbb{Z}: k \leq \xi \}$.
Given functions $f,g\colon [t_0,\infty) \to (0,\infty)$,
we write
\[
f(t) = O \big(
g(t)
\big)\qquad (t \to \infty)
\]
if there exist constants $M>0$ and $t_1 \geq t_0$ such that 
$f(t) \leq M g(t)$ for all $t \geq t_1$.
When $0< \omega_p \leq \omega_q < \infty$,
we denote by $\mathcal{S}(\omega_p,\omega_q)$ the set of 
sequences $(\omega_k)_{k \in \mathbb{N}}$ of positive real numbers
satisfying $
\omega_p\leq \omega_k \leq 
\omega_q$
for all $k \in \mathbb{N}$.

Let $X$ be a Banach space.
We denote by $\mathcal{L}(X)$ the Banach algebra 
of bounded linear operators on $X$.
For a linear operator $A$ on $X$, let $D(A)$, $\sigma(A)$, and $\varrho(A)$
denote the domain, the spectrum, and the resolvent set of $A$, respectively.
Let $-A$ 
be the generator of a bounded $C_0$-semigroup $(e^{-tA})_{t \geq 0}$ 
on $X$.
We define the Cayley transform $V_\omega(A)$ of $A$ with parameter $\omega>0$
by $V_\omega(A) \coloneqq (A - \omega I )(A + \omega I)^{-1}$.
For simplicity of notation, we set $V(A) \coloneqq V_1(A)$.
If $A$ is injective, the fractional power $A^{\alpha}$ of $A$
is defined by the sectoral functional calculus for $\alpha \in \mathbb{R}$;
see \cite[Chapter~3]{Haase2006}.

Let $H$  be a Hilbert space.  We denote the inner product on $H$ by $\langle
\cdot, \cdot \rangle$.
For  a densely-defined linear operator $A$ on $H$,
the Hilbert space adjoint of $A$ is denoted by
$A^*$.

\section{$\mathcal{B}$-calculus}
\label{sec:B_calculus}
In this section, first we recall the notion of the $\mathcal{B}$-calculus and its basic properties.
Then we provide a new operator norm estimate in the $\mathcal{B}$-calculus 
for the negative generator of a 
polynomially stable $C_0$-semigroup on a Hilbert space.
\subsection{Definition and basic facts}
We provide some background material on the $\mathcal{B}$-calculus
and refer the reader to \cite[Sections~2 and 4]{Batty2021} for more details.
Let $\mathcal{B}$ be the algebra of holomorphic functions $f$ 
on $\mathbb{C}_+$ satisfying
\[
\|f\|_{\mathcal{B}_0} \coloneqq \int_{0}^{\infty} \sup_{\eta \in \mathbb{R}}
|f'(\xi+i\eta)| d\xi < \infty.
\]
For all $f \in \mathcal{B}$, the limit $f(\infty) \coloneqq \lim_{\re z \to \infty}
f(z)$ exists in $\mathbb{C}$ and 
\[
\|f\|_{\infty} \coloneqq 
\sup_{z \in \mathbb{C}_+} |f(z)| 
\leq |f(\infty)| + \|f\|_{\mathcal{B}_0}.
\]
Moreover, $\mathcal{B}$ equipped with the norm
\[
\|f\|_{\mathcal{B}} \coloneqq \|f\|_{\infty} + \|f\|_{\mathcal{B}_0},\quad f\in \mathcal{B}
\]
is a Banach algebra. 

Let $\mathrm{M}(\mathbb{R}_+)$ be the 
space of all bounded Borel measures $\mu$ on $\mathbb{R}_+$, endowed with the total variation norm $\|\mu\|_{\mathrm{M}(\mathbb{R}_+)} \coloneqq |\mu|(\mathbb{R}_+)$.
Then  $\mathrm{M}(\mathbb{R}_+)$ is a Banach algebra with respect to
the convolution product. 
We identify $L^1(\mathbb{R}_+)$ with a subalgebra of $\mathrm{M}(\mathbb{R}_+)$.
The Laplace transform of $\mu \in \mathrm{M}(\mathbb{R}_+)$ is the function on $\mathbb{C}_+$ defined by
\[
(\mathcal{L}\mu)(z) \coloneqq \int_{\mathbb{R}_+} e^{-tz} \mu(dt),\quad z \in \mathbb{C}_+.
\]
Define $\mathcal{LM} \coloneqq \{ \mathcal{L}\mu : \mu \in \mathrm{M}(\mathbb{R}_+) \}$.
Then $\mathcal{LM}$ equipped with the norm 
\[
\|\mathcal{L}\mu\|_{
	\mathrm{HP}} \coloneqq \|\mu\|_{\mathrm{M}(\mathbb{R}_+)},\quad  
\mu \in \mathrm{M}(\mathbb{R}_+)
\]
is a Banach algebra. 
If $f \in \mathcal{LM}$, then $f \in \mathcal{B}$ and $
\|f\|_{\mathcal{B}} \leq 2
\|f\|_{
	\mathrm{HP}}$.

We present some examples of $\mathcal{LM}$-functions, which will be used 
to obtain decay estimates for the 
Cayley transform and the inverse of semigroup generators.
\begin{example}
	\label{ex:LM_functions}
	\begin{enumerate}
		\renewcommand{\labelenumi}{(\alph{enumi})}
		\item 
		Let $c >0$ and $d \in \mathbb{R}$. Define 
		\[
		u_{c,d}(z) \coloneqq \frac{z+d}{z+c},\quad z \in \mathbb{C}_+.
		\]
		Then $u_{c,d}$ is the Laplace transform of $\delta_0 - (c-d)e_{c} \in \mathrm{M}(\mathbb{R}_+)$,
		where
		$\delta_0$ is the Dirac measure concentrated at 
		$0$ and $e_{c}(t) \coloneqq e^{-ct}$
		for $t \geq 0$. Hence 
		the $n$-th power 
		$(u_{c,d})^n$ belongs to $\mathcal{LM}$ for every $n \in \mathbb{N}$.
		
		\item 
		Let $\alpha,c >0$. Define 
		\[
		v_{\alpha,c} (z) \coloneqq \frac{1}{(z+c)^{\alpha}},\quad z \in \mathbb{C}_+.
		\]
		By
		the definition of the gamma function $\Gamma$, we obtain
		\[
		\Gamma(\alpha) = \int^{\infty}_0 t^{\alpha-1}e^{-t}dt = 
		(z+c)^{\alpha} \int^{\infty}_0 t^{\alpha-1} e^{-(z+c)t}dt
		\]
		for all positive real numbers $z >0$. Together with 
		the uniqueness theorem for holomorphic functions, this implies that
		$v_{\alpha,c}$ is the Laplace transform of the function $\phi_{\alpha,c} \in L^1(\mathbb{R}_+)$ defined by
		\[
		\phi_{\alpha,c}(t) \coloneqq \frac{t^{\alpha-1} e^{-ct}}{\Gamma(\alpha)},\quad t >0.
		\]
		Therefore, $v_{\alpha,c} \in \mathcal{LM}$.
		
		\item
		For a fixed $t >0$, define 
		\[
		w_t (z) \coloneqq \frac{z}{z+1} e^{-t/z},\quad z \in \mathbb{C}_+.
		\]
		Then
		$w_t$ is the Laplace transform of some $L^1(\mathbb{R}_+)$-function;
		see	 the proof of \cite[Theorem 3.3]{deLaubenfels2009}.
		Hence $w_t \in \mathcal{LM}$.
	\end{enumerate}
\end{example}

Now we introduce a functional calculus for $\mathcal{B}$.
Let $-A$ be the generator of a bounded $C_0$-semigroup $(e^{-tA})_{t \geq 0}$ 
on a Hilbert space $H$. We define $K \coloneqq \sup_{t \geq 0} \|e^{-tA}\|$.
For all $x\in H$, the Plancherel theorem gives 
\[
\int_{-\infty}^{\infty}
\|
(\xi -  i \eta+A)^{-1}x \|^2
d\eta =
2\pi
\int_0^{\infty}
e^{-2\xi t} \|e^{-tA} x\|^2 dt,
\]
and hence
\begin{align}
	\sup_{\xi >0 }
	\xi
	\int_{-\infty}^{\infty}
	\|
	(\xi -  i \eta+A)^{-1}x \|^2
	d\eta 
	\leq \pi K^2 \|x\|^2. 	\label{eq:bouded_int_bound1}
\end{align}
Similarly,
\begin{align}
	\sup_{\xi >0 }
	\xi
	\int_{-\infty}^{\infty}
	\|
	(\xi + i \eta+A^*)^{-1}y \|^2
	d\eta \leq \pi K^2 \|y\|^2
	\label{eq:bouded_int_bound2}
\end{align}
for all $y \in H$.
Combining these estimates with the Cauchy-Schwartz inequality, we obtain
\begin{equation}
	\label{eq:resol_int_bound}
	\sup_{\xi >0 }
	\xi
	\int_{-\infty}^{\infty}
	|\langle
	(\xi - i \eta+A)^{-2}x, y
	\rangle |
	d\eta \leq \pi K^2 \|x\| \, \|y\|
\end{equation}
for all $x,y \in H$.
Let $f \in \mathcal{B}$, and define the linear operator $f(A)$ on $H$ by
\begin{equation}
	\label{eq:fB_def}
	\langle
	f(A)x,y
	\rangle \coloneqq
	\langle 
	f(\infty)x,y
	\rangle 
	- \frac{2}{\pi}
	\int_0^{\infty} \xi
	\int_{-\infty}^{\infty}
	\langle
	(\xi - i\eta+A)^{-2} x, y
	\rangle
	f'(\xi+i \eta) d\eta d\xi
\end{equation}
for $x,y \in H$.
By the estimate \eqref{eq:resol_int_bound}, $f(A)$ is bounded on $H$ and satisfies
\begin{equation}
	\label{eq:fA_bound_Bnorm}
	\|f(A)\| \leq |f(\infty)| + 2 K^2 \|f\|_{\mathcal{B}_0} \leq 2 K^2 \|f\|_{\mathcal{B}}.
\end{equation}
The map
\[
\Phi_A \colon \mathcal{B} \to \mathcal{L}(H),\quad f \mapsto f(A)
\]
is a bounded algebra homomorphism.
We refer to $\Phi_A$ as {\em the $\mathcal{B}$-calculus for $A$}.

Define the map $\Pi_A\colon \mathcal{LM} \to \mathcal{L}(H)$ by
\[
\Pi_A(f)x \coloneqq \int_{\mathbb{R}_+} e^{-tA} x \mu(dt)
\]
for $f=\mathcal{L}\mu \in \mathcal{LM}$ and $x \in H$.
The map $\Pi_A$ is a bounded  algebra homomorphism and is called {\em the Hille-Phillips calculus for $A$}.
The $\mathcal{B}$-calculus extends the Hille-Phillips calculus in the sense that 
$\Phi_A(f) = \Pi_A(f)$ for all $f \in \mathcal{LM}$.

We apply the $\mathcal{B}$-calculus to the functions presented in Example~\ref{ex:LM_functions}.
\begin{example}
	\label{ex:LM_calculus} 
	Let $-A$ be the generator of a bounded $C_0$-semigroup
	on a Hilbert space  such that $0 \in \varrho(A)$.
	We define the functions $u_{c,d},v_{\alpha,c},w_t \in \mathcal{LM}$ as in Example~\ref{ex:LM_functions}.
	\renewcommand{\labelenumi}{(\alph{enumi})}
	\begin{enumerate}
		\item 
		Since $\Pi_A(u_{c,d}) = (A+dI) (A+cI)^{-1}$, we have 
		\[
		(u_{c,d})^n(A) =  ((A+dI) (A+cI)^{-1})^n.
		\]
		\item
		One has $\Pi_A(v_{\alpha,c} ) = (A+cI)^{-\alpha}$ by \cite[Proposition~3.3.5]{Haase2006}, 
		and hence
		$v_{\alpha,c}(A) = (A+cI)^{-\alpha}$. Moreover, since
		\begin{equation}
			\label{eq:inv_fractional}
			A^{-\alpha} =  (A+cI)^{-\alpha} (I+c A^{-1})^{\alpha},
		\end{equation}
		the estimate \eqref{eq:fA_bound_Bnorm} yields
		\[
		\|f(A) A^{-\alpha}\| \leq
		2 K^2\|(I+c A^{-1})^{\alpha}\| \,
		\|fv_{\alpha,c}\|_{\mathcal{B}_0}
		\]
		for all $f \in \mathcal{B}$. We shall use this estimate frequently without comment.
		
		\item 
		The argument in the proof of \cite[Corollary~5.8]{Batty2021} 
		shows that 
		\[
		w_{t}(A) = A(A+I)^{-1} e^{-tA^{-1}}.
		\]
	\end{enumerate}
\end{example}

\subsection{Operator norm estimate for polynomially stable semigroups}
The norm estimate \eqref{eq:fA_bound_Bnorm} can be applied to
every negative generator $A$ of a bounded $C_0$-semigroup on a Hilbert space.
Here we present a sharper norm estimate
for polynomially stable $C_0$-semigroups on Hilbert spaces.
First we recall a useful property on rates of polynomial decay for
bounded $C_0$-semigroups on Banach spaces; see
\cite[Proposition~3.1]{Batkai2006} for the proof.

\begin{proposition}
	\label{prop:frac_normalize} 
	Let $-A$ be the generator of a bounded $C_0$-semigroup $(e^{-tA})_{t\geq 0}$  \sloppy
	on a Banach space $X$ such that $0 \in \varrho(A)$.
	Then the following two statements are equivalent for a fixed $\beta>0$:
	\begin{enumerate}
		\renewcommand{\labelenumi}{\rm  \arabic{enumi}.}
		\item
		$\|e^{-tA}A^{-1}\| = O(t^{-1/\beta})$ as $t \to \infty$.
		\item
		$\|e^{-tA}A^{-\beta q}\| = O(t^{-q})$ as $t \to \infty$
		for some/all $q>0$.
	\end{enumerate}
\end{proposition}

Next we establish a resolvent estimate analogous to 
\eqref{eq:bouded_int_bound1} for
polynomially stable $C_0$-semigroups on Hilbert spaces.
The proof is similar to that of  \cite[Proposition~3.1]{Wakaiki2021JEE}, but
in order to make this paper self-contained,
we give a short argument.
\begin{lemma}
	\label{lem:poly_int}
	Let $-A$ be the generator of a bounded $C_0$-semigroup $(e^{-tA})_{t\geq 0}$  \sloppy
	on a Hilbert space $H$ such that $0 \in \varrho(A)$.
	Then the following two statements are equivalent for fixed $\beta>0$ and 
	$q  \in (0,1/2)$:
	\begin{enumerate}
		\renewcommand{\labelenumi}{\rm  \arabic{enumi}.}
		\item $\|e^{-tA}A^{-1}\| = O(t^{-1/\beta})$ as $t \to \infty$. 
		\item There exists $M>0$ such that for all $x \in H$,
		\begin{equation}
			\label{eq:resol_int_cont_poly}
			\sup_{0<\xi<1}
			\xi^{1-2q } 
			\int_{-\infty}^{\infty}
			\|(\xi+i\eta+A)^{-1}A^{-\beta q } x\|^2 d\eta \leq M\|x\|^2.
		\end{equation}
	\end{enumerate}
\end{lemma}
\begin{proof}
	Let $q  \in (0,1/2)$ and let $K_1 \coloneqq \sup_{t \geq 0} \|e^{-tA}\|$. 
	By Proposition~\ref{prop:frac_normalize},
	the statement~1 holds if and only if
	there exist constants $K_2,t_0 >0$ such that
	\begin{equation}
		\label{eq:A_beta_q_estimate}
		\|e^{-tA}A^{-\beta q }\| \leq \frac{K_2}{t^q }\quad 
		\text{for all $t \geq t_0$}.
	\end{equation}
	
	(Proof of 1 $\Rightarrow$ 2)
	Let $x \in H$.
	The Plancherel theorem gives 
	\begin{equation}
		\label{eq:Plancherel_resol_cond}
		\int_{-\infty}^{\infty} \|(\xi+i\eta+A)^{-1}A^{-\beta q }x\|^2 d\eta =2\pi
		\int_0^{\infty} e^{-2\xi t} \|e^{-tA}A^{-\beta q }x\|^2 dt
	\end{equation}
	for all $\xi >0$.  	
	Using the estimate \eqref{eq:A_beta_q_estimate}, we obtain
	\begin{align*}
		&\int_0^{\infty} e^{-2\xi t} \|e^{-tA} A^{-\beta q }x\|^2 dt \\
		&\qquad =
		\int_0^{t_0} e^{-2\xi t} \|e^{-tA} A^{-\beta q }x\|^2 dt + 
		\int_{t_0}^{\infty} e^{-2\xi t} \|e^{-tA} A^{-\beta q }x\|^2 dt \\
		&\qquad \leq 
		t_0K_1^2  \|A^{-\beta q }\|^2 \, \|x\|^2 + 
		K_2^2 \|x\|^2 \int_{t_0}^{\infty} \frac{e^{-2\xi t}}{t^{2q }}dt
	\end{align*}
	for all $\xi >0$.
	Since
	\[
	\int_{t_0}^{\infty} \frac{e^{-2\xi t}}{t^{2q }}dt \leq  \frac{\Gamma(1-2q )}
	{(2\xi)^{1-2q }},
	\]
	where $\Gamma$ is a gamma function,
	we obtain
	\begin{align}
		&\sup_{0<\xi<1} \xi^{1-2q }
		\int_0^{\infty} e^{-2\xi t} \|e^{-tA} A^{-\beta q }x\|^2 dt \notag \\
		&\qquad \leq \left(t_0K_1^2  \|A^{-\beta q }\|^2 + K_2^2 
		\frac{\Gamma(1-2q )}{2^{1-2q }} \right) \|x\|^2.	\label{eq:int_cond_bound}
	\end{align}
	From \eqref{eq:Plancherel_resol_cond} and \eqref{eq:int_cond_bound},
	we conclude that the statement~2 holds.
	
	(Proof of 2 $\Rightarrow$ 1)
	Using the inverse formula given in \cite[Corollary~III.5.16]{Engel2000}, we have that 
	for all $x \in D(A^2)$, $y \in H$, and $t,\xi >0$, 
	\begin{align*}
		| \langle e^{-tA} x,y \rangle |
		&\leq \frac{e^{\xi t}}{2\pi t}
		\int_{-\infty}^{\infty}
		|
		\langle
		(\xi+i\eta+A)^{-2}x, y 
		\rangle
		| d\eta;
	\end{align*}
	see also \cite[Corollary~8.14]{Batty2021JFA} for the inversion formula.
	Together with \eqref{eq:bouded_int_bound2}, \eqref{eq:resol_int_cont_poly},
	and the Cauchy-Schwartz inequality, this estimate implies that
	there exists a constant $M_1>0$ such that 
	\[
	\|e^{-tA} A^{-\beta q } x\| 
	\leq \frac{M_1  e^{\xi t}}{\xi^{1-q} \, t} \|x\|
	\]
	for all $x \in D(A^2)$, $t>0$, and $0 < \xi < 1$.
	Since $D(A^2)$ is dense in $H$, 
	setting $\xi = 1/t$ yields
	\[
	\|e^{-tA} A^{-\beta q } \|  \leq 
	\frac{e M_1 }{ t^{q }}
	\]
	for all $t > 1$.
	Thus, $\|e^{-tA}A^{-1}\| = O(t^{-1/\beta})$ is obtained.
\end{proof}

Let $q  >0$, and define
\[
\|f\|_{\mathcal{B}_0,q } \coloneqq \int_0^{\infty} \psi_{q } (\xi)
\sup_{\eta \in \mathbb{R}}  |f'(\xi+i\eta)| d\xi,\quad f \in \mathcal{B},
\] 
where
\begin{equation}
	\label{eq:psi_def}
	\psi_{q }(\xi) \coloneqq 
	\begin{cases}
		\xi^{q }, & 0<\xi < 1, \\
		1, & \xi \geq 1.
	\end{cases}
\end{equation}
Since $0 < \psi_q(\xi) \leq 1$ for all $\xi >0$, we have
$\|f\|_{\mathcal{B}_0,q} \leq \|f\|_{\mathcal{B}_0}$ for all $f \in \mathcal{B}$.
When $-A$ is the generator of a polynomially stable $C_0$-semigroup with parameter $\beta>0$,
the operator norm
for $f(A)A^{-\beta q }$ can be upper-bounded by using 
$\|f\|_{\mathcal{B}_0,q }$.
\begin{proposition}
	\label{prop:poly_B_bound}
	Let 
	$-A$ be the generator of 
	a polynomially stable  $C_0$-semi\-group  with parameter $\beta >0$ on a Hilbert space $H$. Then, 
	for all $q  \in (0,1/2)$, there exists $M>0$
	such that 
	\begin{equation}
		\label{eq:fA_bound_Bnorm_poly}
		\|f(A) A^{-\beta q }\| \leq 
		\|A^{-\beta q } \|  \, |f(\infty)|  + 
		M \|f\|_{\mathcal{B}_0,q }
	\end{equation}
	for all $f \in \mathcal{B}$.
\end{proposition}
\begin{proof}
	Let $q  \in (0,1/2)$ and let $K \coloneqq \sup_{t \geq 0} \|e^{-tA}\|$.
	For all $x, y \in H$ and
	$f\in \mathcal{B}$,
	\begin{align}
		&
		\left| \int_0^{\infty} \xi
		\int_{-\infty}^{\infty}
		\langle
		(\xi - i \eta + A)^{-2}x,y
		\rangle f'(\xi+i\eta) d\eta d\xi \right| 
		\notag 
		\\
		&~~\leq
		\int_0^{\infty} \psi_{q } (\xi) \sup_{\eta \in \mathbb{R}}  |f'(\xi+i\eta)|
		\frac{\xi}{\psi_{q }(\xi)}
		\int_{-\infty}^{\infty}
		|\langle
		(\xi - i \eta + A)^{-2}x,y
		\rangle  |d\eta d\xi.
		\label{eq:fAx_x}
	\end{align}
	To the right-hand side,
	we apply
	Lemma~\ref{lem:poly_int} for $0<\xi < 1$
	and the estimate \eqref{eq:bouded_int_bound1}  for $\xi \geq 1$.
	For $x \in D(A^{\beta q })$,
	we replace $x$ by  $A^{\beta q}x$  in
	Lemma~\ref{lem:poly_int} and use the inequality
	$\|x\| \leq \|A^{-\beta q}\|\, \|A^{\beta q} x\|$ in \eqref{eq:bouded_int_bound1}.
	Then we see that
	there exists $M_1>0$ such that 
	\begin{equation}
		\label{eq:poly_int_bound}
		\sup_{\xi>0}
		\frac{\xi}{\psi_{q }(\xi)^2}
		\int_{-\infty}^{\infty}
		\|
		(\xi - i \eta + A)^{-1}x \|^2
		d\eta 
		\leq M_1^2 \|A^{\beta q } x\|^2
	\end{equation}
	for all $x \in D(A^{\beta q })$.
	The estimates~\eqref{eq:bouded_int_bound2} and
	\eqref{eq:poly_int_bound} together with
	the Cauchy-Schwartz inequality 
	imply
	\begin{equation}
		\label{eq:sup_xi_estimate}
		\sup_{\xi >0 }
		\frac{\xi}{\psi_{q }(\xi)}
		\int_{-\infty}^{\infty}
		|\langle
		(\xi - i \eta + A)^{-2}x,y
		\rangle  |d\eta  \leq \sqrt{\pi} KM_1 \|A^{\beta q } x \|  \, \|y\|
	\end{equation}
	for all $x \in D(A^{\beta q })$ and $y \in H$. 
	Applying
	the estimates \eqref{eq:fAx_x} and \eqref{eq:sup_xi_estimate}
	to the definition \eqref{eq:fB_def} of $f(A)$, we derive
	\[
	|\langle
	f(A)x,y
	\rangle
	| \leq 
	|f(\infty)| \, \|x\| \, \|y\| + 
	\frac{2KM_1 }{\sqrt{\pi}} \|f\|_{\mathcal{B}_0,q }  \, \|A^{\beta q } x \| \, \|y\|
	\]
	for all $x \in D(A^{\beta q })$, $y \in H$, and $f \in \mathcal{B}$.
	Thus,
	\[
	\|f(A) A^{-\beta q }\| \leq 
	\|A^{-\beta q } \| \, |f(\infty)| + 
	\frac{2KM_1 }{\sqrt{\pi}} \|f\|_{\mathcal{B}_0,q }
	\]
	for all $f \in \mathcal{B}$.
\end{proof}

In Sections~\ref{sec:CT} and \ref{sec:IG},
we will derive operator norm estimates for the Cayley transform and
the inverse of a semigroup generator from
estimates of $\|f\|_{\mathcal{B}_0}$ or $\|f\|_{\mathcal{B}_0,q}$ 
for the corresponding $\mathcal{B}$-functions $f$, by using
the inequalities \eqref{eq:fA_bound_Bnorm} and \eqref{eq:fA_bound_Bnorm_poly}.
The inequality
\eqref{eq:fA_bound_Bnorm_poly} will be applied in the following way.
\begin{example}
	\label{ex:bound_poly}
	Let $q \in (0, 1/2)$ and let
	$-A$ be the generator of 
	a polynomially stable  $C_0$-semigroup with parameter $\beta >0$ on a Hilbert space.
	For 
	$\alpha,c >0$,
	we define $v_{\alpha,c} \in \mathcal{LM}$ as in Example~\ref{ex:LM_functions}.(b).
	Using \eqref{eq:inv_fractional},
	we obtain
	\[
	f(A)A^{-\alpha-\beta q} = (fv_{\alpha,c})(A)A^{-\beta q} ( I + cA^{-1} )^{\alpha}
	\]
	for all $f \in \mathcal{B}$. Therefore, 
	Proposition~\ref{prop:poly_B_bound} shows that
	there exists $M>0$ such that 
	\[
	\|f(A)A^{-\alpha-\beta q} \| 
	\leq M  \|fv_{\alpha,c}\|_{\mathcal{B}_0,q}
	\]
	for all $f \in \mathcal{B}$.
\end{example}

\section{Estimates for Cayley transforms}
\label{sec:CT}
In this section, first we study the
asymptotic behavior of $\|V(A)^n A^{-\alpha}\|$ for $\alpha >0$
when $-A$ is the generator of 
a polynomially stable $C_0$-semigroup on a Hilbert space.
Next we turn our attention to exponentially stable $C_0$-semigroups on
Hilbert spaces
and 
the case of variable parameters $\| ( \prod_{k=1}^nV_{\omega_k}(A) )A^{-\alpha} \|$.
\subsection{Case of constant parameters}
When $-A$ is the generator of a polynomially stable $C_0$-semigroup on a Hilbert space,
the following 
estimate for $\|V(A)^n A^{-\alpha}\|$ holds
without the assumption that 
$(V(A)^n)_{n \in \mathbb{N}_0}$ is bounded.
Recall that $\lfloor \xi \rfloor$ is the largest integer less than or 
equal to $\xi \in \mathbb{R}$.
\begin{theorem}
	\label{thm:CT_bound_poly}
	Let 
	$-A$ be the generator of 
	a polynomially stable $C_0$-semigroup 
	with parameter $\beta >0$
	on a Hilbert space $H$. Then,
	for all $\alpha >0$, 
	\begin{equation}
		\label{eq:CT_bound}
		\|V(A)^n A^{-\alpha}\| = O\left(
		\frac{(\log n)^{k+1} }{n^{\alpha / (2+\beta)}}
		\right)\quad (n \to \infty),\quad \text{where~}
		k \coloneqq \left\lfloor
		\frac{2\alpha}{2+\beta}
		\right \rfloor.
	\end{equation}
\end{theorem}

Let $n \in \mathbb{N}$ and $p >0$.
Define
\begin{equation}
	\label{eq:fn_CT_poly}
	f_{n,2p}(z) \coloneqq \frac{(z-1)^n}{(z+1)^{n+2p}},\quad z \in \mathbb{C}_+.
\end{equation}
Then
$f_{n,2p} \in \mathcal{LM}$; 
see  Example~\ref{ex:LM_functions}.
To prove Theorem~\ref{thm:CT_bound_poly}, we need a preliminary
estimate
for $\|f_{n,2p}\|_{\mathcal{B}_0, q }$.
\begin{proposition}
	\label{prop:f_n_norm}
	For $n \in \mathbb{N}$ and $p >0$, 
	define the function $f_{n,2p}$ by \eqref{eq:fn_CT_poly}.
	Then, for all $q  \in (0,1/2)$, 
	\begin{align}
		\label{eq:f_n_norm_poly}
		\|f_{n,2p}\|_{\mathcal{B}_0, q } = 
		\begin{cases}
			O\left(
			\dfrac{1}{n^p}
			\right),& p <  q , \vspace{3pt}\\
			O\left(
			\dfrac{\log n}{n^p}
			\right), 
			& p=  q ,
			\qquad (n \to \infty). \vspace{3pt}\\
			O\left(
			\dfrac{1}{n^{ q }}
			\right),
			& p >  q,
		\end{cases}
	\end{align}
\end{proposition}

\begin{proof}
	Step~1: Let $n \in \mathbb{N}$. Fix $p >0$ and $ q  \in (0,1/2)$. 
	Since
	\[
	f_{n,2p}'(z) = - 2p \frac{(z-1)^n}{(z+1)^{n+2p+1}} + 2n \frac{(z-1)^{n-1}}{(z+1)^{n+2p+1}},
	\]
	we have
	\begin{equation}
		\label{eq:f_gh_bound_CT}
		\sup_{\eta \in \mathbb{R}} |f_{n,2p}'(\xi+i \eta )| \leq 
		2p \sup_{s \geq 0 }g_{n,p}(\xi,s) + 2n \sup_{s \geq 0 } h_{n,p}(\xi,s),
	\end{equation}
	where 
	\begin{align*}
		g_{n,p}(\xi,s) &\coloneqq 
		\frac{ \big((\xi-1)^2 + s\big)^{n/2}  }{ \big((\xi+1)^2 + s\big)^{(n+1)/2 + p} },\\
		h_{n,p}(\xi,s) &\coloneqq 
		\frac{ \big((\xi-1)^2 + s\big)^{(n-1)/2}  }{ \big((\xi+1)^2 + s\big)^{(n+1)/2+p} }
	\end{align*}
	for $\xi>0$ and $s \geq 0$.
	In Steps~2 and 3 below, we will show that 
	\begin{align}
		\int_0^{\infty}\psi_{ q }(\xi) \sup_{s \geq 0} g_{n,p}(\xi,s) d\xi
		&= 
		\begin{cases}
			O\left(
			\dfrac{1}{n^p}
			\right),& p \leq  q , \vspace{3pt}\\
			O\left(
			\dfrac{1}{n^{ q }}
			\right),& p >  q,
		\end{cases}
		\qquad (n \to \infty)
		\label{eq:g_n_estimate}
	\end{align}
	and
	\begin{align}
		\int_0^{\infty} \psi_{ q }(\xi) \sup_{s \geq 0} h_{n,p}(\xi,s) d\xi 
		&= 
		\begin{cases}
			O\left(
			\dfrac{1}{n^{p+1}}
			\right),& p< q , \vspace{3pt}\\
			O\left(
			\dfrac{\log n}{n^{p+1}}
			\right),& p= q ,\qquad (n \to \infty). \vspace{3pt}\\
			O\left(
			\dfrac{1}{n^{ q +1}}
			\right),& p> q,  \\
		\end{cases}
		\label{eq:h_n_estimate}
	\end{align}
	Combining the estimates \eqref{eq:f_gh_bound_CT}--\eqref{eq:h_n_estimate},
	we obtain the desired conclusion \eqref{eq:f_n_norm_poly}.
	
	Step~2: The aim of this step is to show the estimate \eqref{eq:g_n_estimate}
	for $g_{n,p}$.
	Fix $\xi > 0$, and 
	define 
	\[
	G_{\xi}(s) \coloneqq g_{n,p}(\xi,s)^2 =
	\frac{ \big((\xi-1)^2 + s\big)^n  }{ \big((\xi+1)^2 + s\big)^{n+ 2p+1} }
	\]
	for $s \geq 0$.
	Then
	\[
	G_{\xi}'(s) = \frac{\big((\xi-1)^2+s\big)^{n-1}}{\big((\xi+1)^2+s\big)^{n+2p+2}} 
	\chi (s),
	\]
	where
	\[
	\chi (s) \coloneqq n \big(
	(\xi+1)^2 + s
	\big)
	- (n+2p+1) \big((\xi -1)^2 + s \big).
	\]
	Since 
	\[
	\chi (s) =
	2(2p + 1 + 2n)\xi- (2p+1)(s+\xi^2+1),
	\]
	the equation $\chi (s) =0$ has the following solution:
	\[
	s_1 \coloneqq -\xi^2-1 + 2\left(
	1 + \frac{2n}{2p+1}
	\right)\xi.
	\]
	Moreover, 
	the solution $s_1$ satisfies
	$s_1 \geq 0$ if and only if
	$
	\xi_0
	\leq \xi \leq \xi_1,
	$
	where
	\begin{align*}
		\xi_0 &\coloneqq  \xi_0(n) \coloneqq 
		1+ \frac{2n}{2p+1} - 
		\frac{2\sqrt{n(n+2p+1)}}{2p+1}, \\
		\xi_1 &\coloneqq \xi_1(n) \coloneqq 
		1+ \frac{2n}{2p+1} +
		\frac{2\sqrt{n(n+2p+1)}}{2p+1}.
	\end{align*}
	Note that $\xi_1 >1$. Hence $\xi_0  = 1/\xi_1  < 1$.
	We also have
	\begin{equation}
		\label{eq:g_np_bound_CT}
		\sup_{s \geq 0} g_{n,p}(\xi,s) =
		\begin{cases}
			\displaystyle \sqrt{ G_{\xi} (s_1)}, &  \xi \in [\xi_0, \xi_1], \vspace{3pt}\\
			\displaystyle  \sqrt{ G_{\xi}(0)}, &  \xi \notin [\xi_0, \xi_1].
		\end{cases} 
	\end{equation}
	
	Since routine calculations show that 
	\begin{align*}
		(\xi-1)^2+s_1 &=\frac{4n}{2p+1}\xi, \\
		(\xi+1)^2+s_1 &=  \frac{4(n+2p+1)}{2p + 1} \xi,
	\end{align*}
	we obtain
	\begin{align*}
		G_{\xi}  (s_1) &= 
		\left(\frac{n}{n+2p+1}  \right)^{n}
		\left(  \frac{2p + 1} {4(n+2p+1)} \right)^{2p+1 }
		\frac{1}{\xi^{2p+1}}.
	\end{align*}
	Combining this with \eqref{eq:g_np_bound_CT}, we obtain
	\begin{align*}
		&\int_{\xi_0}^{\xi_1} \psi_{ q } (\xi) \sup_{s \geq 0} g_{n,p}(\xi,s) d\xi
		\\
		&\qquad =
		\left(\frac{n}{n+2p+1}  \right)^{n/2}
		\left(  \frac{2p + 1} {4(n+2p+1)} \right)^{p+1/2 }
		\int_{\xi_0}^{\xi_1} \frac{\psi_{ q } (\xi)}{\xi^{p+ 1/2}} d\xi.
	\end{align*}
	By the definition \eqref{eq:psi_def} of $\psi_{ q }$,
	\[
	\int_{\xi_0}^{\xi_1} \frac{\psi_{ q } (\xi)}{\xi^{p+ 1/2}} d\xi =
	\int_{\xi_0}^{1} \frac{1}{\xi^{p -  q  + 1/2}} d\xi + 
	\int_{1}^{\xi_1} \frac{1}{\xi^{p + 1/2}} d\xi.
	\]
	If $p \not=1/2$ and $p\not= q+1/2 $, then
	\begin{align*}
		\int_{\xi_0}^{\xi_1} \frac{\psi_{ q } (\xi)}{\xi^{p+ 1/2}} d\xi  =
		\left(
		\frac{1}{ q -p+1/2} - \frac{1}{1/2-p}
		\right) -
		\frac{\xi_0^{ q -p+1/2}}{q -p+1/2}
		+
		\frac{\xi_1^{1/2-p}}{1/2-p},
	\end{align*}
	and hence we have from $\xi_0(n) = O(n^{-1})$ and $\xi_1(n) = O(n)$ as $n \to \infty$
	that
	\begin{equation}
		\int_{\xi_0}^{\xi_1} \psi_{ q } (\xi) \sup_{s \geq 0} g_{n,p}(\xi,s) d\xi
		=
		\begin{cases}
			O\left(
			\dfrac{1}{n^{2p}}
			\right),& p < \dfrac{1}{2}, \vspace{3pt}\\
			O\left(
			\dfrac{1}{n^{p+1/2}}
			\right),& \dfrac{1}{2} < p <  q  + \dfrac{1}{2}, \vspace{3pt}\\
			O\left(
			\dfrac{1}{n^{ q +1}}
			\right),& p >  q  +\dfrac{1}{2}
		\end{cases}
		\label{eq:gnp_poly1}
	\end{equation}
	as $n \to \infty$. Similarly,
	if $p = 1/2$ or $p=  q  + 1/2$, then
	\begin{equation}
		\label{eq:gnp_poly2}
		\int_{\xi_0}^{\xi_1} \psi_{ q } (\xi) \sup_{s \geq 0} g_{n,p}(\xi,s) d\xi = 
		O\left(
		\frac{\log n}{n^{p+1/2}}
		\right)
	\end{equation}
	as $n \to \infty$.
	On the other hand, we derive from \eqref{eq:g_np_bound_CT} that
	\begin{align*}
		\int_{\xi_1}^{\infty} \psi_{ q }(\xi)  \sup_{s \geq 0} g_{n,p}(\xi,s) d\xi 
		= \int_{\xi_1}^{\infty}  \frac{(\xi-1)^n}{(\xi+1)^{n+2p+1}} d\xi
		\leq \int_{\xi_1}^{\infty}  \frac{1}{\xi^{2p+1}} d\xi
	\end{align*}
	and
	\begin{align*}
		\int_{0}^{\xi_0} \psi_{ q }(\xi)  \sup_{s \geq 0} g_{n,p}(\xi,s) d\xi 
		= \int_{0}^{\xi_0}  \frac{\xi^{ q  }(1-\xi)^n}{(\xi+1)^{n+2p+1}} d\xi
		\leq \int_{0}^{\xi_0} \xi^{ q  }d\xi .
	\end{align*}
	Therefore,
	\begin{equation}
		\label{eq:gnp_poly3}
		\int_{\xi_1}^{\infty} \psi_{ q }(\xi)  \sup_{s \geq 0} g_{n,p}(\xi,s) d\xi 
		= O\left(
		\frac{1}{n^{2p}}
		\right)
	\end{equation}
	and
	\begin{equation}
		\label{eq:gnp_poly4}
		\int_{0}^{\xi_0} \psi_{ q }(\xi)  \sup_{s \geq 0} g_{n,p}(\xi,s) d\xi =
		O\left(
		\frac{1}{n^{ q +1}}
		\right)
	\end{equation}
	as $n \to \infty$.
	Thus, the estimate \eqref{eq:g_n_estimate} for $g_{n,p}$ holds, and it
	is actually less sharp than the estimate obtained from the above argument.
	
	Step~3:
	We immediately obtain the estimate \eqref{eq:h_n_estimate} for $h_{n,p}$,
	by replacing 
	$n$ with $n-1$ and $p$ with $p+1/2$ in 
	the estimates~\eqref{eq:gnp_poly1}--\eqref{eq:gnp_poly4}.
\end{proof}

We are well prepared to obtain the estimate \eqref{eq:CT_bound}
for $\|V(A)^n A^{-\alpha}\|$.
\begin{proof}[Proof of Theorem~\ref{thm:CT_bound_poly}]
	Let $p \in (0,1/2)$
	and define the function $f_{n,2p}$
	by \eqref{eq:fn_CT_poly} for $n \in \mathbb{N}$.
	By Example~\ref{ex:bound_poly} in the case $\alpha = 2p$, $c=1$, $q = p$, and
	\[
	f(z) = \left( \frac{z-1}{z+1} \right)^n,
	\]
	there exists a constant $M>0$ such that 
	\[
	\|V(A)^n A^{-(2+\beta)p}\| \leq M \|f_{n,2p}\|_{\mathcal{B}_0,p}
	\]
	for all $n \in \mathbb{N}$.
	From  the estimate \eqref{eq:f_n_norm_poly}
	with $p =  q$, we have 
	\begin{equation}
		\label{eq:CT_p}
		\|V(A)^n A^{-(2+\beta)p}\| = O\left(
		\frac{\log n}{n^p}
		\right)
	\end{equation}
	as $n \to \infty$. Hence,  the desired estimate \eqref{eq:CT_bound}
	holds in the case $k = \lfloor 2\alpha/(2+\beta) \rfloor =0$.
	
	Next we consider the case $k = \lfloor 2\alpha/(2+\beta) \rfloor \geq 1$. To this end,
	choose $\alpha \geq (2+\beta)/2$ arbitrarily. Let  $k \in \mathbb{N}$ and
	$0 \leq \delta_0 < (2+\beta)/2$
	satisfy
	\[
	\alpha = \frac{k(2+\beta)}{2} + \delta_0.
	\]
	Then there exists $\varepsilon >0$ such that 
	\[
	0 < \delta \coloneqq \delta_0 + \frac{\varepsilon k}{2} < \frac{2+\beta}{2}.
	\]
	By definition, 
	\[
	\alpha = k \gamma+ \delta, \quad \text{where~}
	\gamma  \coloneqq \frac{2+\beta - \varepsilon}{2},
	\]
	and hence
	\begin{align*}
		\|V(A)^n A^{-\alpha}\|
		&\leq
		\left( \max_{0\leq \ell \leq k} \|V(A)^{\ell}\|\right)
		\|V(A)^{\lfloor n/(k+1) \rfloor } A^{-\gamma } \|^k \,
		\|V(A)^{\lfloor n/(k+1) \rfloor } A^{-\delta } \|
	\end{align*}
	for all $n \in \mathbb{N}$.
	Using the estimate \eqref{eq:CT_p}, we obtain
	\begin{align*}
		\|V(A)^{\lfloor n/(k+1) \rfloor } A^{-\gamma } \|=
		O\left(
		\frac{\log n}{n^{\gamma/(2+\beta)} } 
		\right)
	\end{align*}
	and
	\begin{align*}
		\|V(A)^{\lfloor n/(k+1) \rfloor } A^{-\delta } \| = 
		O\left(
		\frac{\log n}{n^{\delta/(2+\beta)}} 
		\right)
	\end{align*}
	as $n \to \infty$.
	Thus, the desired estimate \eqref{eq:CT_bound} holds.
\end{proof}

\begin{remark}
	\label{rem:p_gamma}
	In the proof of Theorem~\ref{thm:CT_bound_poly},
	we have used the estimate  \eqref{eq:f_n_norm_poly}
	with $p =  q $ for $\|f_{n,2p}\|_{\mathcal{B}_0, q }$.
	This is because the estimate  \eqref{eq:f_n_norm_poly}
	with $p \not=  q $ yields a worse decay rate.
	To see this, 
	assume that the $C_0$-semigroup 
	$(e^{-tA})_{t \geq 0}$ is polynomially stable 
	with parameter $\beta >0$.
	The estimate  \eqref{eq:f_n_norm_poly} with $p <  q \, (< 1/2) $ gives
	\begin{equation}
		\label{eq:p_smaller_gamma}
		\|V(A)^n A^{-2p-\beta q }\| = O\left(
		\frac{1}{n^{p}}
		\right)\qquad (n \to \infty).
	\end{equation}
	If $2p+\beta  q  = 1$ and $p <  q $, then
	\[
	p < \frac{1}{2+\beta}.
	\]
	Hence
	the estimate \eqref{eq:p_smaller_gamma} with $2p+\beta  q  = 1$  is worse than 
	the estimate \eqref{eq:CT_bound} with $\alpha = 1$. A similar argument can be applied to the estimate  \eqref{eq:f_n_norm_poly} with  $p >  q $.
\end{remark}

\subsection{Case of variable parameters}
Here we study the case where the parameter $\omega$ of the Cayley transform $V_{\omega}(A)$ 
varies.
First we improve the previous estimate developed in 
\cite[Theorem~3.3]{Wakaiki2023IEOT}
for exponentially stable $C_0$-semigroups.
Next we consider polynomially stable $C_0$-semigroups with normal generators
and present a similar improved estimate over the one given in 
\cite[Proposition~3.6]{Wakaiki2023IEOT}.
\subsubsection{Exponentially stable semigroups}
When the $C_0$-semigroup $(e^{-tA})_{t \geq 0}$ is exponentially stable,
we give an estimate for $\| (\prod_{k=1}^n V_{\omega_k}(A) ) A^{-\alpha} \|$
without assuming the boundedness of $(e^{-tA^{-1}})_{t \geq 0}$.
Before that,
we define the function $F_{\alpha}$ on $\mathbb{N}$ by
\begin{equation}
	\label{eq:Fa_def}
	F_{\alpha} (n) \coloneqq
	\begin{cases}
		\log n, & \alpha = 0, \vspace{3pt}\\
		\dfrac{1}{n^{\alpha/2}}, & \alpha >0.
	\end{cases}
\end{equation}
Recall  that 
for $0<\omega_p\leq \omega_q < \infty$,
we denote by $\mathcal{S}(\omega_p,\omega_q)$
the set of 
sequences $(\omega_k)_{k \in \mathbb{N}}$ of positive real numbers
satisfying $
\omega_p\leq \omega_k \leq 
\omega_q$
for all $k \in \mathbb{N}$.
\begin{theorem}
	\label{thm:CT_exp_variable}
	Let $-A$ be the generator of an exponentially stable $C_0$-semi\-group
	on a Hilbert space $H$. Then,
	for all $\alpha \geq 0$ and $0< \omega_p \leq \omega_q < \infty$,
	there exist constants $M>0$ and $n_0 \in \mathbb{N}$ such that 
	\begin{equation}
		\label{eq:VA_exp_variable}
		\left\| \left(\prod_{k=1}^n V_{\omega_k}(A)\right) A^{-\alpha}  \right\| \leq M 
		F_{\alpha}(n)
	\end{equation}
	for all $n \geq n_0$ and $(\omega_k)_{k \in \mathbb{N}} \in 
	\mathcal{S}(\omega_p,\omega_q)$,
	where the function $F_{\alpha}$ is as in \eqref{eq:Fa_def}.
\end{theorem}

When $(e^{-tA})_{t \geq 0}$ is an exponentially stable $C_0$-semigroup,
$-A + cI$ generates a bounded $C_0$-semigroup for some sufficiently small constant $c>0$.
To prove Theorem~\ref{thm:CT_exp_variable} by means of the $\mathcal{B}$-calculus, 
we consider the function $f_{n,\alpha,(\omega_k)}$ defined by
\begin{equation}
	\label{eq:fn_def_exp_variable}
	f_{n,\alpha,(\omega_k)}(z) 
	\coloneqq
	\frac{1}{(z+\omega_p+c)^{\alpha}}
	\prod_{k=1}^n
	\frac{z-\omega_k+c}{z+\omega_k+c},\quad z \in \mathbb{C}_+,
\end{equation}
where $0 < \omega_p \leq \omega_q < \infty$ and $(\omega_k)_{k \in \mathbb{N}} \in 
\mathcal{S}(\omega_p,\omega_q)$.
As seen in Example~\ref{ex:LM_functions},
we have $f_{n,\alpha,(\omega_k)} \in \mathcal{LM}$.
The following result gives an estimate for $\|f_{n,\alpha,(\omega_k)}\|_{\mathcal{B}_0}$.
\begin{proposition}
	\label{prop:fn_exp}
	Let $\alpha \geq 0$,
	$c> 0$, and $0 < \omega_p \leq \omega_q < \infty$. Then
	there exist constants $M>0$ and $n_0 \in \mathbb{N}$ such that 
	the function $f_{n,\alpha,(\omega_k)} $ defined by \eqref{eq:fn_def_exp_variable} 
	satisfies
	\begin{equation}
		\label{eq:f_n_norm}
		\|f_{n,\alpha,(\omega_k)}\|_{\mathcal{B}_0} \leq MF_{\alpha}(n)
	\end{equation}
	for all $n \geq n_0$ and $(\omega_k)_{k \in \mathbb{N}} \in 
	\mathcal{S}(\omega_p,\omega_q)$,
	where the function $F_{\alpha}$ is as in \eqref{eq:Fa_def}.
\end{proposition}

The difficulty in estimating
$\|f_{n,\alpha,(\omega_k)}\|_{\mathcal{B}_0}$
is that the parameter $\omega_k$ can vary on $[\omega_p,\omega_q]$,
which complicates the computation of 
$
\sup_{\eta \in \mathbb{R}} |f_{n,\alpha,(\omega_k)}'(\xi+i \eta)|.
$
To circumvent this difficulty, we show that the function on $[\omega_p,\omega_q]$ defined by
\[
\omega \mapsto \left|
\frac{(\xi -\omega+c) + i \eta}{(\xi +\omega+c)+i \eta}
\right|
\]
takes the maximum value at $\omega = \omega_p$ for all $\xi > 0$ and $\eta \in \mathbb{R}$ under a suitable condition on the constant $c$.

\begin{lemma}
	\label{lem:each_bound}
	Let 
	$c> 0$ and $0 < \omega_p \leq \omega_q < \infty$
	satisfy  $c^2 \geq \omega_p \omega_q$. 
	If $\omega_p \leq \omega \leq \omega_q$, then
	\[
	\frac{(\xi-\omega+c)^2+s}{(\xi+\omega+c)^2+s} \leq 
	\frac{(\xi-\omega_p+c)^2+s}{(\xi+\omega_p+c)^2+s} 
	\] 
	for all $\xi >0$ and $s \geq 0$. 
\end{lemma}
\begin{proof}
	Fix $\xi >0$ and $s \geq 0$.  Set $\zeta \coloneqq \xi + c$. Define
	\[
	\phi_{\zeta,s}(\omega) \coloneqq \frac{(\zeta-\omega)^2+s}{(\zeta+\omega)^2+s},
	\quad \omega >0.
	\]
	Then
	\[
	\phi_{\zeta,s}'(\omega) = \frac{4\zeta(
		\omega^2 - \zeta^2-s 
		)}{\big((\zeta+\omega)^2+s\big)^2}.
	\]
	Therefore, 
	$\omega_0 \coloneqq \sqrt{\zeta^2+s}$ satisfies $\phi_{\zeta,s}'(\omega_0) = 0$,
	and $\phi_{\zeta,s}$ decreases on the interval $(0,\omega_0)$
	and increases on the interval $(\omega_0,\infty)$.
	This implies that for all $\omega
	\in [\omega_p, \omega_q]$,
	\[
	\phi_{\zeta,s} (\omega) \leq 
	\max\{
	\phi_{\zeta,s} (\omega_p),\,\phi_{\zeta,s} (\omega_q)
	\}.
	\]
	
	It remains to show that if $c^2 \geq \omega_p \omega_q$, 
	then
	\begin{equation}
		\label{eq:min_max_estimate}
		\phi_{\zeta,s} (\omega_p) \geq \phi_{\zeta,s} (\omega_q).
	\end{equation}
	This can be proved by a direct calculation.
	We have
	\[
	\phi_{\zeta,s} (\omega_p) - \phi_{\zeta,s} (\omega_q) =
	\frac{
		\chi_{\zeta,s}
	}{ \big((\zeta+\omega_p\big)^2 + s)\big((\zeta+\omega_q)^2 + s\big)},
	\]
	where
	\[
	\chi_{\zeta,s} \coloneqq
	\big((\zeta-\omega_p)^2 + s\big)\big((\zeta+\omega_q)^2 + s\big) - 
	\big((\zeta-\omega_q)^2 + s\big) \big((\zeta+\omega_p)^2 + s\big).
	\]
	One can rewrite $\chi_{\zeta,s}$ as
	\begin{align*}
		\chi_{\zeta,s} &= 
		\big((\zeta-\omega_p)^2(\zeta+\omega_q)^2 -
		(\zeta-\omega_q)^2(\zeta+\omega_p)^2\big) \\
		&\qquad +
		\big(
		(\zeta-\omega_p)^2 + (\zeta+\omega_q)^2 -
		(\zeta-\omega_q)^2 - (\zeta+\omega_p)^2
		\big)s.
	\end{align*}
	The coefficient of $s$ satisfies
	\begin{align*}
		(\zeta-\omega_p)^2 + (\zeta+\omega_q)^2 -
		(\zeta-\omega_q)^2 - (\zeta+\omega_p)^2
		=
		4(\omega_q - \omega_p) \zeta \geq 0.
	\end{align*}
	Moreover, a routine calculation shows that 
	\begin{align*}
		(\zeta-\omega_p)^2(\zeta+\omega_q)^2 -
		(\zeta-\omega_q)^2(\zeta+\omega_p)^2 = 4(\omega_q - \omega_p)
		\zeta (\zeta^2 - \omega_p\omega_q) .
	\end{align*}
	Substituting $\zeta = \xi + c$, we have from the condition $c^2 \geq \omega_p \omega_q$ that 
	\[
	\zeta^2 - \omega_p\omega_q = \xi^2+2c \xi + (c^2 - \omega_p\omega_q) \geq 0.
	\]
	Thus, 
	the inequality \eqref{eq:min_max_estimate} holds.
\end{proof}

Next we give auxiliary estimates to be used 
after Lemma~\ref{lem:each_bound} is applied.
\begin{lemma}
	\label{lem:g_h_bound}
	Let $n \in \mathbb{N}$, $\alpha \geq 0$, and $\omega,c >0$.
	Define
	\begin{align}
		g_{n,\alpha,\omega}(\xi,s) &\coloneqq \frac{\big((\xi-\omega + c )^2+s \big)^{n/2}}{\big((\xi+\omega + c)^2 + s\big)^{(n+\alpha+1)/2}},
		\label{eq:g_naw}\\
		h_{n,\alpha,\omega}(\xi,s) &\coloneqq \frac{\big((\xi-\omega + c )^2+s\big)^{(n-1)/2}}{\big((\xi+\omega + c )^2 + s\big)^{(n+\alpha+1)/2}}
		\label{eq:h_naw}
	\end{align}
	for $\xi >0$ and $s \geq 0$. Then 
	\begin{align}
		\int_0^{\infty} \sup_{s \geq 0} g_{n,\alpha,\omega}(\xi,s)  d\xi &= 
		\begin{cases}
			O\left( \dfrac{1}{n^{\alpha}} \right), & 0 \leq \alpha < 1, \vspace{3pt} \\
			O\left( \dfrac{\log n}{n} \right), &  \alpha = 1, 
			\hspace{30pt} (n \to \infty) \vspace{3pt} \\
			O\left( \dfrac{1}{n^{(\alpha+1)/2}} \right), & \alpha > 1,
		\end{cases}
		\label{eq:g_bound} 
	\end{align}
	and
	\begin{align}
		\int_0^{\infty} \sup_{s \geq 0} h_{n,\alpha,\omega}(\xi,s) d\xi &= 
		\begin{cases}
			O\left( \dfrac{\log n}{n} \right), & \alpha = 0, \vspace{3pt} \\
			O\left( \dfrac{1}{n^{\alpha/2+1}} \right), & \alpha >0,
		\end{cases}
		\qquad (n \to \infty).
		\label{eq:h_bound}
	\end{align}
\end{lemma}
\begin{proof}
	It suffices to show  the estimate \eqref{eq:g_bound}  for $g_{n,\alpha,\omega}$.
	Indeed, the estimate \eqref{eq:h_bound} for $h_{n,\alpha,\omega}$
	is obtained  by replacing $n$ with $n-1$ and $\alpha$
	with $\alpha+1$ in the estimate \eqref{eq:g_bound} for $g_{n,\alpha,\omega}$. 
	
	Let $n \in \mathbb{N}$, $\alpha \geq 0$, and $\omega >0$. 
	To obtain the estimate for $g_{n,\alpha,\omega}$, fix $\xi >0$ and
	define 
	\[
	G_{\xi}(s) \coloneqq g_{n,\alpha,\omega}(\xi,s)^2 = 
	\frac{\big((\xi-\omega + c )^2+s \big)^{n}}{\big((\xi+\omega + c)^2 + s\big)^{n+\alpha+1}}
	\]
	for $s \geq 0$.
	Then
	\[
	G_{\xi}'(s) = \frac{\big((\xi-\omega+c)^2+s\big)^{n-1}}
	{\big((\xi+\omega+c)^2+s\big)^{n+\alpha+2}} \chi (s),
	\]
	where
	\[
	\chi (s) \coloneqq n \big((\xi+\omega+c)^2+s\big) - 
	(n+\alpha+1) \big((\xi-\omega+c)^2+s \big).
	\]
	We have
	\[
	\chi  (s) =
	(\alpha+1)
	\left(
	-\xi^2 + 2
	\left(
	\frac{2\omega n}{\alpha+1} + \omega - c
	\right) \xi
	+
	\frac{4\omega c n}{\alpha+1} - (\omega - c)^2 - s
	\right).
	\]
	Therefore, $\chi(s_1) = 0$, where $s_1$ is defined by
	\[
	s_1 \coloneqq -\xi^2 + 2
	\left(
	\frac{2\omega n}{\alpha+1} + \omega - c
	\right) \xi
	+
	\frac{4\omega c n}{\alpha+1} - (\omega- c)^2 .
	\] 
	Since $\omega, c >0$, 
	we obtain
	\[
	\frac{4\omega c n}{\alpha+1} - (\omega- c)^2  >  0
	\]
	for all sufficiently large $n \in \mathbb{N}$. Hence,
	there exists $n_1 \in \mathbb{N}$ such that 
	the equation
	\[
	-\xi^2 + 2
	\left(
	\frac{2\omega n}{\alpha+1} + \omega - c
	\right) \xi
	+
	\frac{4\omega c n}{\alpha+1} - (\omega- c)^2 = 0
	\]
	has a unique positive solution $\xi_1 = \xi_1(n)$ for all $n \geq n_1$, and 
	\[
	\xi_1(n) 
	=
	\left(
	\frac{2\omega n}{\alpha+1} + \omega - c
	\right) + \frac{2\omega \sqrt{n (n + \alpha + 1)}}{\alpha + 1}.
	\]
	
	Let $n \geq n_1$.
	We obtain
	\begin{equation}
		\label{eq:g_na_bound}
		\sup_{s \geq 0} g_{n,\alpha,\omega}(\xi,s) =
		\begin{cases}
			\displaystyle \sqrt{ G_{\xi}(s_1)}, & \xi \leq \xi_1, \vspace{3pt}\\
			\displaystyle \sqrt{ G_{\xi}(0)}, & \xi> \xi_1.
		\end{cases} 
	\end{equation}
	Routine calculations show that 
	\begin{align*}
		(\xi-\omega + c )^2 + s_1 &= \frac{4\omega n }{\alpha+1} (\xi+c),\\
		(\xi+\omega + c)^2 + s_1 &= \frac{4\omega (n+\alpha+1)}{\alpha+1} (\xi+c).
	\end{align*}
	Hence we obtain
	\begin{align*}
		G_{\xi}(s_1) =
		\left(
		\frac{n}{n+\alpha+1}
		\right)^{n}
		\left( \frac{\alpha+1}{4\omega (n+\alpha+1)}
		\right)^{\alpha+1}
		\frac{1}{(\xi+c)^{\alpha+1}}.
	\end{align*}
	This and  \eqref{eq:g_na_bound} give
	\begin{align*}
		&\int_0^{\xi_1} \sup_{s \geq 0} g_{n,\alpha,\omega}(\xi,s) d\xi \\
		&\qquad =
		\left(
		\frac{n}{n+\alpha+1}
		\right)^{n/2}
		\left( \frac{\alpha+1}{4\omega (n+\alpha+1)}
		\right)^{(\alpha+1)/2}
		\int^{\xi_1}_0 \frac{1}{(\xi+c)^{(\alpha+1)/2}} d\xi.
		\label{eq:int_below_x1}
	\end{align*}
	Noting that $\xi_1(n) = O(n)$ as $n \to \infty$,
	we have
	\begin{equation}
		\label{eq:g_bound1_exp} 
		\int_0^{\xi_1} \sup_{s \geq 0} g_{n,\alpha,\omega}(\xi,s)  d\xi =
		\begin{cases}
			O\left( \dfrac{1}{n^{\alpha}} \right), & \alpha < 1, \vspace{3pt} \\
			O\left( \dfrac{\log n}{n} \right), &  \alpha = 1,  \vspace{3pt} \\
			O\left( \dfrac{1}{n^{(\alpha+1)/2}} \right), & \alpha > 1
		\end{cases}
	\end{equation}
	as $n \to \infty$.
	We also see from \eqref{eq:g_na_bound} that 
	\begin{align*}
		\int_{\xi_1}^{\infty} \sup_{s \geq 0} g_{n,\alpha,\omega}(\xi,s) d\xi 
		&=
		\int_{\xi_1}^{\infty}  \frac{(\xi-\omega+c)^{n}}{(\xi+\omega+c)^{n+\alpha +1}}d\xi  
		\leq
		\int_{\xi_1}^{\infty}  \frac{1}{(\xi+\omega+c)^{\alpha +1}}d\xi.
	\end{align*}
	Hence
	\begin{equation}
		\label{eq:g_bound2_exp}
		\int_{\xi_1}^{\infty} \sup_{s \geq 0} g_{n,\alpha,\omega}(\xi,s) d\xi = O\left(
		\frac{1}{n^{\alpha}}
		\right)
	\end{equation}
	as $n \to \infty$. 
	Combining the estimates 
	\eqref{eq:g_bound1_exp}  and \eqref{eq:g_bound2_exp}, we obtain
	the estimate \eqref{eq:g_bound} for $g_{n}$.
\end{proof}

Now we are in a position to prove Proposition~\ref{prop:fn_exp}
and Theorem~\ref{thm:CT_exp_variable}.
\begin{proof}[Proof of Proposition~\ref{prop:fn_exp}]
	Define $\tilde \omega_p \coloneqq \min \{ \omega_p,\,c^2/\omega_q \} $. Then
	$\mathcal{S}(\omega_p,\omega_q) \subset \mathcal{S}(\tilde \omega_p,\omega_q)$ 
	and 
	$c^2 \geq \tilde \omega_p \omega_q$. 
	Replacing $\omega_p$ with $\tilde \omega_p$ in the definition \eqref{eq:fn_def_exp_variable}
	of $f_{n,\alpha,(\omega_k)}$,
	we define the function $\tilde f_{n,\alpha,(\omega_k)}$ by
	\[
	\tilde f_{n,\alpha,(\omega_k)}(z) 
	\coloneqq
	\frac{1}{(z+\tilde \omega_p+c)^{\alpha}}
	\prod_{k=1}^n
	\frac{z-\omega_k+c}{z+\omega_k+c},\quad z \in \mathbb{C}_+,
	\]
	where $(\omega_k)_{k \in \mathbb{N}} \in 
	\mathcal{S}(\tilde \omega_p,\omega_q)$.
	Then, for all $(\omega_k)_{k \in \mathbb{N}} \in 
	\mathcal{S}(\omega_p,\omega_q)$,
	\[
	f_{n,\alpha,(\omega_k)} =
	r_{\alpha}
	\tilde f_{n,\alpha,(\omega_k)},
	\]
	where 
	\[
	r_{\alpha}(z) \coloneqq
	\left(
	\frac{z+\tilde \omega_p+c}{z+\omega_p+c}
	\right)^{\alpha},\quad z \in \mathbb{C}_+.
	\]
	Since $r_1 \in \mathcal{LM}$ (see Example~\ref{ex:LM_functions}.(a)),
	we have from 
	\cite[Proposition 2.2]{Batty2021} that $r_{\alpha} \in \mathcal{B}$.
	
	Now we apply Lemmas~\ref{lem:each_bound} and \ref{lem:g_h_bound} to
	$\tilde f_{n,\alpha,(\omega_k)}$.
	The derivative of $\tilde f_{n,\alpha,(\omega_k)}$ is given by
	\begin{align*}
		\tilde f_{n,\alpha,(\omega_k)}'(z) &= 
		\frac{-\alpha}{(z+\tilde\omega_p+c)^{\alpha+1}} 
		\prod_{k=1}^n
		\frac{z-\omega_k+c}{z+\omega_k+c} \\
		&\qquad + 
		\frac{2}{(z+\tilde\omega_p+c)^{\alpha}} 
		\sum_{\ell=1}^n	
		\frac{\omega_\ell}{(z+\omega_\ell+c)^2}
		\prod_{k=1,k\not=\ell}^n
		\frac{z-\omega_k+c}{z+\omega_k+c}.
	\end{align*}
	Let $\xi >0$, $\eta \in \mathbb{R}$, and 
	$s \coloneqq \eta^2 \geq 0$.
	Since $c^2 \geq \tilde \omega_p \omega_q$,
	Lemma~\ref{lem:each_bound} shows that
	\begin{align*}
		|\tilde f_{n,\alpha,(\omega_k)}'(\xi+i \eta)| \leq 
		\alpha g_{n,\alpha,\tilde\omega_p}(\xi,s)  + 2\omega_q n h_{n,\alpha,\tilde\omega_p}(\xi,s) 
	\end{align*}
	for all $(\omega_k)_{k \in \mathbb{N}} \in 
	\mathcal{S}(\tilde\omega_p,\omega_q)$,
	where $g_{n,\alpha,\tilde\omega_p}$ and $h_{n,\alpha,\tilde\omega_p}$ are defined by \eqref{eq:g_naw} and \eqref{eq:h_naw} with
	$\omega=\tilde \omega_p$, respectively.
	From Lemma~\ref{lem:g_h_bound}, we see that there exist
	constants $\tilde M>0$ and $n_0 \in \mathbb{N}$ such that 
	\[
	\|\tilde f_{n,\alpha,(\omega_k)}\|_{\mathcal{B}_0} \leq \tilde M F_{\alpha}(n)
	\]
	for all $n \geq n_0$ and $(\omega_k)_{k \in \mathbb{N}} \in 
	\mathcal{S}(\tilde \omega_p,\omega_q)$.
	Since 
	$\mathcal{S}(\omega_p,\omega_q) \subset \mathcal{S}(\tilde \omega_p,\omega_q)
	$ and since
	\[
	\|f_{n,\alpha,(\omega_k)}\|_{\mathcal{B}} \leq 
	\|r_{\alpha}\|_{\mathcal{B}}\,  \|\tilde f_{n,\alpha,(\omega_k)}\|_{\mathcal{B}},
	\]
	for all $(\omega_k)_{k \in \mathbb{N}} \in 
	\mathcal{S}(\omega_p,\omega_q)$, the constant $M \coloneqq 
	2\|r_{\alpha}\|_{\mathcal{B}}
	\tilde M $ satisfies
	\[
	\|f_{n,\alpha,(\omega_k)}\|_{\mathcal{B}_0} \leq M F_{\alpha}(n)
	\]
	for all $n \geq n_0$ and $(\omega_k)_{k \in \mathbb{N}} \in 
	\mathcal{S}(\omega_p,\omega_q)$.
\end{proof}

\begin{proof}[Proof of Theorem~\ref{thm:CT_exp_variable}]
	Since the $C_0$-semigroup $(e^{-tA})_{t \geq 0}$ is exponentially stable,
	there exist constants $K \geq 1$ and $c >0$ such that 
	\[
	\|e^{-tA}\| \leq K e^{-c t}
	\]
	for all $t \geq 0$.
	Define $B \coloneqq  A - c I$. Then
	$-B$ is the generator of a bounded $C_0$-semigroup $(e^{-tB})_{t \geq 0}$, and 
	$\sup_{t \geq 0} \|e^{-tB}\| \leq K$.
	
	Define the function $f_{n,\alpha,(\omega_k)}$ by \eqref{eq:fn_def_exp_variable}.
	Then
	\begin{align*}
		f_{n,\alpha,(\omega_k)}(B) &= \left( 
		\prod_{k=1}^n
		(B-\omega_k I + c I) (B+\omega_k I + c I)^{-1} \right)
		(B+\omega_pI + c  I)^{-\alpha}\\
		&=
		\left(  \prod_{k=1}^n 
		V_{\omega_k}(A)
		\right)
		(A+\omega_pI)^{-\alpha}.
	\end{align*}
	The inequality \eqref{eq:fA_bound_Bnorm} gives
	\[
	\|f_{n,\alpha,(\omega_k)}(B)\| \leq  2K^2  \|f_{n,\alpha,(\omega_k)}\|_{\mathcal{B}_0}.
	\]
	Combining this with
	Proposition~\ref{prop:fn_exp}, we see that
	there exist constants $M>0$ and $n_0 \in \mathbb{N}$ such that 
	the desired estimate \eqref{eq:VA_exp_variable} holds
	for all $n \geq n_0$ and $(\omega_k)_{k \in \mathbb{N}} \in 
	\mathcal{S}(\omega_p,\omega_q)$.
\end{proof}

We provide some remarks on the optimality of the estimates given in
Proposition~\ref{prop:fn_exp} and 
Theorem~\ref{thm:CT_exp_variable}.
\begin{remark}
	\label{rem:optimality1}
	The estimate \eqref{eq:f_n_norm} for $\|f_{n,\alpha,(\omega_k)}\|_{\mathcal{B}_0}$ 
	with $\alpha >0$
	cannot be  improved in general even
	in the constant case $\omega_k \equiv 1$.
	To see this, fix
	$c>0$, and define
	\begin{equation}
		\label{eq:f_na_rem}
		f_{n,\alpha}(z) \coloneqq \frac{(z-1+c)^n}{(z+1+c)^{n+\alpha}},\quad z \in \mathbb{C}_+,
	\end{equation}
	where $n \in \mathbb{N}$ and $\alpha >0$.
	Then we have
	\begin{equation}
		\label{eq:f_na_lower_bound}
		\|f_{n,\alpha} \|_{\mathcal{B}_0} \geq \|f_{n,\alpha}\|_{\infty} = \sup_{\eta \in \mathbb{R}} 
		\left|
		f_{n,\alpha}(i \eta)
		\right|.
	\end{equation}
	Since
	\[
	|f_{n,\alpha}(i \sqrt{n})|^2 = \frac{1}{\big((1+c)^2+n\big)^{\alpha}} \left( 
	\frac{(1-c)^2+n }{(1+c)^2+n}
	\right)^n,
	\]
	there exist $M_1>0$ and $n_1 \in \mathbb{N}$ such that 
	\begin{equation}
		\label{eq:f_na_point_lower_bound}
		|f_{n,\alpha}(i \sqrt{n})| \geq \frac{M_1}{n^{\alpha/2}}
	\end{equation}
	for all $n \geq n_1$.
	The estimates \eqref{eq:f_na_lower_bound} and \eqref{eq:f_na_point_lower_bound}
	yield
	\[
	\|f_{n,\alpha}\|_{\mathcal{B}_0} \geq \frac{M_1}{n^{\alpha/2}}
	\]
	for all $n \geq n_1$.
\end{remark}
\begin{remark}
	\label{rem:optimality2}
	One cannot in general improve
	the operator norm estimate \eqref{eq:VA_exp_variable} for $\left\| \left(\prod_{k=1}^n V_{\omega_k}(A)\right) A^{-\alpha}  \right\| $ with $\alpha >0$ either. 
	Let $B$ be a self-adjoint operator on a Hilbert space $H$, and assume that 
	\[
	\sigma(B) \supset \{ \sqrt{n} : n \in \mathbb{N} \text{~and~} n \geq n_2 \}
	\]
	for some $n_2 \in \mathbb{N}$.
	Fix $c>0$ and
	define $A \coloneqq iB + cI$. Then $-A$ is the generator of 
	an exponentially stable $C_0$-semigroup on $H$.
	For $n \in \mathbb{N}$ and $\alpha >0$,
	define the function $f_{n,\alpha} \in \mathcal{B}$ by \eqref{eq:f_na_rem}.
	Since
	\[
	V(A) = (iB -I + cI) (iB +I +cI)^{-1},
	\]
	the spectral mapping theorem (see, e.g., \cite[Theorem~2.7.8]{Haase2006})
	shows that
	\begin{align*}
		\|V(A)^n (A+I)^{-\alpha}\| =
		\|f_{n,\alpha}(iB) \| 
		= \sup_{\lambda \in \sigma(iB)} |f_{n,\alpha}(\lambda)|.
	\end{align*}
	By the estimate \eqref{eq:f_na_point_lower_bound},
	\[
	\|V(A)^n A^{-\alpha}\|  \geq \frac{M_2}{n^{\alpha/2}}
	\]
	for all $n \geq \max\{n_1,n_2  \}$, where $M_2 \coloneqq M_1/ \|( I +A^{-1})^{-\alpha}\|$.
\end{remark}

\begin{remark}
	The argument in the proof of 
	Proposition~\ref{prop:fn_exp}
	gives
	a non-sharp constant $M >0$ for the estimate \eqref{eq:f_n_norm}, 
	although the decay rate $n^{-\alpha/2}$ cannot be improved in general.
	In the proof,
	we have replaced $\omega_p$ by $c^2/\omega_q$ when
	$c^2/\omega_q < \omega_p$. 
	Consequently, the constant $M$ there
	may allow
	a unnecessary large variation of $(\omega_k)_{k\in\mathbb{N}}$
	when
	$c$ is small and $\omega_q$ is large. Finding a sharp constant $M$ for given constants $\omega_p$ and $\omega_q$
	is left as a topic for further research.
\end{remark}

\subsubsection{Polynomially stable semigroups with normal generators}
Using spectral properties of normal operators,
one can also obtain
an estimate of $\|(\prod_{k=1}^n V_{\omega_k}(A) ) A^{-\alpha}\|$
for the normal generator $-A$ of a 
polynomially stable $C_0$-semigroup.
A similar result holds for multiplication operators on $L^p$-spaces and
spaces of continuous functions that vanish at infinity.

\begin{proposition}
	\label{prop:poly_normal}
	Consider the following two cases:
	\begin{enumerate}
		\renewcommand{\labelenumi}{\rm (\alph{enumi})}
		\item
		Let 
		$(e^{-tA})_{t\geq 0}$ be 
		a bounded $C_0$-semigroup  on a Hilbert space 
		$H$ such that 
		$A$ is normal.
		\item 
		Let $\Omega$ be a locally compact Hausdorff space and let $\mu$
		be a $\sigma$-finite regular Borel measure on $\Omega$. Assume that 
		either
		\begin{enumerate}
			\renewcommand{\labelenumii}{\rm (\roman{enumii})}
			\item $X \coloneqq L^p(\Omega,\mu)$ for $1\leq p < \infty$ and 
			$\phi \colon \Omega \to \mathbb{C}$ is measurable with essential range in
			$\mathbb{C}_+$; or that
			\item $X \coloneqq  C_0(\Omega)$ and $\phi\colon \Omega \to \mathbb{C}$
			is continuous with range in $\mathbb{C}_+$.
		\end{enumerate}
		Let $A$ be the multiplication operator induced by $\phi$ on $X$, i.e,
		$Af = \phi f$ with domain $D(A) \coloneqq \{f \in X\colon \phi f \in X\}$.
	\end{enumerate}
	In both cases {\em (a)} and {\em (b)},
	assume that $\|e^{-tA} (I+A)^{-1} \| = O(t^{-1/\beta})$ as $t\to \infty$ for
	some $\beta>0$. Then, for all $\alpha>0$ and $0< \omega_p \leq \omega_q < \infty$,
	there exist constants $M>0$ and $n_0 \in \mathbb{N}$ such that 
	\begin{equation}
		\label{eq:VA_poly_variable}
		\left\| \left( \prod_{k=1}^n V_{\omega_k}(A)  \right) A^{-\alpha}  \right\| 
		\leq \frac{M}{n^{\alpha/(2+\beta)}}
	\end{equation}
	for all $n \geq n_0$ and $(\omega_k)_{k \in \mathbb{N}} \in 
	\mathcal{S}(\omega_p,\omega_q)$.
\end{proposition}

\begin{proof}
	Step~1:
	Let $n \in \mathbb{N}$, $\alpha >0$, and  $(\omega_k)_{k \in \mathbb{N}} \in \mathcal{S}(\omega_p,\omega_q)$.
	Define
	\[
	f_{n,\alpha,(\omega_k)}(z) \coloneqq
	\left( 
	\prod_{k=1}^n \frac{z - \omega_k}{z+\omega_k}
	\right)
	z^{-\alpha} 
	\]
	for $z \in \mathbb{C}_+$.
	The estimate $\|e^{-tA} (I+A)^{-1} \| = O(t^{-1/\beta})$ as $t\to \infty$
	implies that $\sigma(A)$ does not intersect with the imaginary axis $i \mathbb{R}$; 
	see \cite[Theorem~1.1]{Batty2008}.
	Hence $\sigma(A) \subset \mathbb{C}_+$ and $A$ has a bounded inverse
	in both cases (a) and (b). Moreover,
	\begin{equation}
		\label{eq:norm_spectrum}
		\left\| \left(  \prod_{k=1}^n V_{\omega_k}(A) \right) A^{-\alpha} \right\|  =
		\sup_{\lambda \in \sigma(A)} |f_{n,\alpha,(\omega_k)}(\lambda)|,
	\end{equation}
	where we used 
	the spectral mapping theorem (see, e.g., \cite[Theorem~2.7.8]{Haase2006}) for the case (a) and \cite[Propositions I.4.2 and I.4.10]{Engel2000} for the case (b); see also
	the proof of \cite[Proposition~4.5]{Wakaiki2021JEE}, which provides a detailed derivation
	of \eqref{eq:norm_spectrum}
	in the case $\omega_k \equiv 1$.
	
	By \eqref{eq:norm_spectrum}, it suffices to show that there exist constants 
	$M>0$ and $n_0 \in \mathbb{N}$ such that 
	\begin{equation}
		\label{eq:f_nao_bound_normal}
		\sup_{\lambda \in \sigma(A)} |f_{n,\alpha,(\omega_k)}(\lambda)| \leq 
		\frac{M}{n^{\alpha/(2+\beta)}}
	\end{equation}
	for all $n \geq n_0$ and $(\omega_k)_{k \in \mathbb{N}} \in 
	\mathcal{S}(\omega_p,\omega_q)$.
	We have from Proposition~\ref{prop:frac_normalize} that
	$\|e^{-tA} A^{-\beta} \| = O(t^{-1})$ as $t \to \infty$.
	Hence, in both cases (a) and (b), 
	there exist $\delta, C>0$ such that 
	$|\im \lambda| \geq C (\re \lambda)^{-1/\beta}$ for all $\lambda \in \sigma(A)$
	with $\re \lambda \leq \delta$; see
	\cite[Propositions~4.1 and 4.2]{Batkai2006}.
	Define
	\begin{align*}
		\Omega_1 := 
		\{
		\lambda \in \sigma(A):
		\re \lambda \leq \delta
		\},\quad 
		\Omega_2 := \sigma(A) \setminus \Omega_1.
	\end{align*}
	We divide the proof of the estimate \eqref{eq:f_nao_bound_normal}
	into two cases:
	$\lambda \in \Omega_1$
	and $\lambda \in \Omega_2$.
	The following observation will be useful for later purposes: 
	For all $\lambda \in \sigma(A)$ and $\omega >0$,
	\begin{equation}
		\label{eq:1pmlambda}
		\left|
		\frac{\lambda - \omega}{\lambda + \omega}
		\right|^2 = 
		\frac{(\re \lambda - \omega)^2 + |\im \lambda|^2}
		{(\re \lambda + \omega)^2 + |\im \lambda|^2} = 
		1 - \frac{4 \omega \re \lambda}{|\lambda+\omega|^2}.
	\end{equation}
	
	Step~2:
	First we assume that $\lambda \in \Omega_1$. 
	We have
	\[
	|\lambda + \omega_q|^2 \leq C_0 |\im \lambda|^2
	\]
	for some constant $C_0 \geq 1$ independent 
	of $\lambda$.
	Combining this estimate with $\re \lambda \geq C^{\beta} / |\im \lambda|^{\beta}$, we obtain
	\[
	\frac{4 \omega \re \lambda}{|\lambda+\omega|^2} \geq
	\frac{4C^{\beta} \omega_p}{C_0 |\im \lambda|^{\beta+2}}
	\]
	for all $\omega \in [\omega_p,\omega_q]$.
	Then \eqref{eq:1pmlambda} gives that for all $\omega \in [\omega_p,\omega_q]$,
	\[
	\left|
	\frac{\lambda-\omega}{\lambda+\omega}
	\right|^2 \leq 1 - \frac{C_1}{|\im \lambda|^{\beta+2}},
	\quad \text{where~}
	C_1 \coloneqq
	\frac{4 C^{\beta} \omega_p }{C_0}.
	\]
	Hence
	\begin{equation}
		\label{eq:im_lam_bound}
		|f_{n,\alpha,(\omega_k)}(\lambda)| \leq 
		\frac{1}{|\im \lambda|^{\alpha}}
		\left(
		1 - \frac{C_1}{|\im \lambda|^{\beta+2}}
		\right)^{n/2}
	\end{equation}
	for all $(\omega_k)_{k \in \mathbb{N}} \in 
	\mathcal{S}(\omega_p,\omega_q)$.
	
	Let $a,b,c>0$, and
	define the function $g_n$ by
	\[
	g_n(s) \coloneqq \frac{1}{s^a} \left(
	1 - \frac{c}{s^b}
	\right)^{n/2},\quad s \geq  c^{1/b}.
	\]
	Then
	\[
	g_n'(s) = \frac{1}{s^{a+1}}\left(
	1 - \frac{c}{s^b}
	\right)^{n/2-1} \frac{
		ac+bcn/2 - as^b
	}{s^b }.
	\]
	for all $s > c^{1/b}$.
	Since $g_n$ takes the maximum value at
	\[
	s = \left(c + \frac{bc}{2a}n \right)^{1/b} >  c^{1/b},
	\]
	we obtain
	\begin{equation}
		\label{eq:gn_estimate}
		\sup_{s \geq c^{1/b}} g_n(s)
		= \left(c + \frac{bc}{2a}n \right)^{-a/b}
		\left(
		1 - \frac{2a}{2a+bn}
		\right)^{n/2}.
	\end{equation}
	This and the estimate \eqref{eq:im_lam_bound} show that 
	there exist $M_1>0$ and $n_1 \in \mathbb{N}$ 
	such that 
	\begin{equation}
		\label{eq:fn_estimate1}
		\sup_{\lambda \in \Omega_1}
		|f_{n,\alpha,(\omega_k)}(\lambda)| \leq 
		\frac{M_1}{n^{\alpha/(\beta+2)}}
	\end{equation}
	for all $n \geq n_1$ and $(\omega_k)_{k \in \mathbb{N}} \in \mathcal{S}(\omega_p,\omega_q)$.
	
	Step~3:
	Next we assume that $\lambda \in \Omega_2$. Since $\re \lambda > \delta$,
	we obtain
	\[
	\frac{4 \omega \re \lambda}{|\lambda+\omega|^2} \geq
	\frac{4 \delta \omega}{ ( |\lambda|+\omega )^{2}}
	\geq 
	\frac{4 \delta \omega_p}{ (1 +\omega_q/\delta )^2|\lambda|^{2}}
	\]
	for all $\omega \in [\omega_p,\omega_q]$.
	Together with \eqref{eq:1pmlambda}, this implies that for all
	$\omega \in [\omega_p,\omega_q]$,
	\begin{equation*}
		\left|
		\frac{\lambda - \omega}{\lambda + \omega}
		\right|^2 \leq 1 - \frac{C_2}{|\lambda|^{2}},
		\quad \text{where~}
		C_2 \coloneqq
		\frac{4\delta\omega_p}{(1+\omega_q/\delta)^2}.
	\end{equation*}
	From this estimate, we derive
	\begin{equation*}
		|f_{n,\alpha,(\omega_k)}(\lambda)| \leq 
		\frac{1}{|\lambda|^\alpha}\left(
		1 - \frac{C_2}{|\lambda|^{2}}
		\right)^{n/2}
	\end{equation*}
	for all $(\omega_k)_{k \in \mathbb{N}} \in 
	\mathcal{S}(\omega_p,\omega_q)$.
	Using the estimate \eqref{eq:gn_estimate}  again, we see that 
	there exist $M_2>0$ and $n_2\in \mathbb{N}$ 
	such that 
	\begin{equation}
		\label{eq:fn_estimate2}
		\sup_{\lambda \in \Omega_2}
		|f_{n,\alpha,(\omega_k)}(\lambda)| \leq 
		\frac{M_2}{n^{\alpha/2}} 
	\end{equation}
	for all $n \geq n_2$ and $(\omega_k)_{k \in \mathbb{N}} \in \mathcal{S}(\omega_p,\omega_q)$.
	Thus, the desired estimate \eqref{eq:f_nao_bound_normal} is obtained
	from \eqref{eq:fn_estimate1} and \eqref{eq:fn_estimate2}.
\end{proof}

\section{Estimates for inverse generators}
\label{sec:IG}
In this section, we are interested in the asymptotic behavior of $\|e^{-tA^{-1}}A^{-\alpha}\|$
for $\alpha >0$, which has been studied in the
Banach space case \cite{Zwart2007, deLaubenfels2009} and the
Hilbert space case \cite{Wakaiki2023IEOT} as mentioned in 
Section~\ref{sec:introduction}.
First we show that $\|e^{-tA^{-1}}A^{-\alpha}\|$ is uniformly bounded for $t \geq 0$
if $-A$ is the generator of a bounded $C_0$-semigroup on a Hilbert space and 
has a bounded inverse. Next we estimate the decay rate of $\|e^{-tA^{-1}}A^{-\alpha}\|$
for a polynomially stable $C_0$-semigroup $(e^{-tA})_{t \geq 0}$ on a Hilbert space.
\subsection{Bounded semigroups}
We present an estimate of $\|e^{-tA^{-1}}A^{-\alpha}\|$ for a 
bounded $C_0$-semigroup $(e^{-tA})_{t \geq 0}$ such that 
the generator $-A$ has a bounded inverse.
\begin{theorem}
	\label{thm:Inv_bound}
	Let $-A$ be the generator of a bounded $C_0$-semigroup \sloppy
	on a Hilbert space $H$ such that $0 \in \varrho(A)$. Then, for all $\alpha >0$,
	\begin{equation}
		\label{eq:inv_bounded}
		\sup_{t \geq 0} \|e^{-tA^{-1} } A^{-\alpha}\| < \infty.
	\end{equation}
\end{theorem}

Let $t >0$ and $\alpha \geq 0$.
To prove Theorem~\ref{thm:Inv_bound}, we
consider the function $f_{t,\alpha}$ on $\mathbb{C}_+$ defined by
\begin{equation}
	\label{eq:ft_def_inv}
	f_{t,\alpha}(z) 
	\coloneqq
	\frac{ze^{-t/z}}{(z+1)^{\alpha+1}} ,\quad z \in \mathbb{C}_+.
\end{equation}
One has $f_{t,\alpha} \in \mathcal{LM}$; see Example~\ref{ex:LM_functions}.
In \cite[Lemma~5.6]{Batty2021JFA}, 
the  estimate in the case $\alpha = 0$,
\begin{equation}
	\label{eq:phi_t_bound}
	\|f_{t,0}\|_{\mathcal{B}_0} = O(\log t)\qquad (n \to \infty),
\end{equation}
has been obtained. 
The next proposition gives an estimate for $\|f_{t,\alpha}\|_{\mathcal{B}_0}$ 
in the case $\alpha >0$.
\begin{proposition}
	\label{prop:ft_bound}
	Let $t,\alpha >0$ and
	define the function $f_{t,\alpha}$ by \eqref{eq:ft_def_inv}.
	Then 
	\begin{equation}
		\label{eq:f_t_norm_bound}
		\sup_{t > 0}\|f_{t,\alpha}\|_{\mathcal{B}_0} < \infty.
	\end{equation}
\end{proposition}

The estimate \eqref{eq:f_t_norm_bound} can be proved in a way similar to
the estimate \eqref{eq:phi_t_bound}.
We need the following auxiliary result.
\begin{lemma}
	\label{lem:g_inv}
	Let $p >0$, and define 
	\begin{equation}
		\label{eq:g_ta_inv}
		g_{t,p}(\xi,s) \coloneqq \frac{e^{-t\xi/(\xi^2+s)}}{ \big((\xi+1)^2+s \big)^{p+1/2} \sqrt{\xi^2+s}} 
	\end{equation}
	for $t,\xi >0$ and $s \geq 0$. Then 
	\begin{align}
		\label{eq:g_inv_bound1}
		\sup_{t > 0} t
		\int_0^{\infty} \sup_{s \geq 0} g_{t,p}(\xi,s) d\xi< \infty.
	\end{align}
\end{lemma}
\begin{proof}
	Define $\zeta \coloneqq t \xi >0$ and 
	\[
	G_{\zeta} (\tau) \coloneqq \frac{e^{-\zeta/\tau}}{(\tau+1)^{p+1/2}  \sqrt{\tau}},
	\quad \tau >0.
	\]
	Then
	\begin{equation}
		\label{eq:v_zeta_bound}
		g_{t,p}(\xi,s)
		\leq G_{\zeta} (\xi^2+s)
	\end{equation}
	for all $t,\xi >0$ and $s \geq 0$.
	We have
	\[
	G_{\zeta}'(\tau) = \frac{-2(p + 1)\tau^2+(2\zeta-1) \tau + 2\zeta }{2  (\tau+1)^{p+3/2}\tau^{5/2} } e^{-\zeta/\tau}.
	\]
	There exists a unique positive solution $\tau_1 = \tau_1(\zeta)$ of the equation 
	\[
	-2(p+ 1)\tau^2+(2\zeta-1) \tau + 2\zeta = 0, 
	\]
	and
	\[
	\tau_1(\zeta) = \frac{(2\zeta - 1) + \sqrt{(2\zeta - 1)^2 + 16\zeta (p+1)} }{4(p+1)}.
	\]
	Since
	\[
	\sqrt{(2\zeta - 1)^2 + 16\zeta(p+1)}  \geq 
	\sqrt{(2\zeta - 1)^2 + 16 \zeta}
	\geq 1,
	\]
	we have
	\[
	\tau_1 \geq c_1 \zeta,
	\quad \text{where~} c_1\coloneqq \frac{1}{2(p+1)}.
	\]
	This yields
	\begin{equation}
		\label{eq:sup_v_zeta_bound}
		\sup_{s \geq 0}G_{\zeta} (\xi^2+s) \leq 
		\sup_{\tau \geq 0 } G_{\zeta}(\tau) =
		G_{\zeta}(\tau_1) \leq \frac{1}{(c_1\zeta + 1)^{p+1/2} \sqrt{c_1\zeta}}
	\end{equation}
	for all $t,\xi >0$.
	This together with \eqref{eq:v_zeta_bound} implies that
	\begin{align*}
		t\int_0^{\infty} \sup_{s \geq 0} g_{t,p}(\xi,s) d\xi 
		&\leq 
		t\int_0^{\infty} \frac{1}{(c_1t\xi + 1)^{p+1/2} \sqrt{c_1t\xi}} d\xi \\
		&=  \frac{1}{c_1}
		\int_0^{\infty} \frac{1}{(\xi + 1)^{p+1/2} \sqrt{\xi}} d\xi < \infty
	\end{align*}
	for all $t >0$.
	Thus, the estimate \eqref{eq:g_inv_bound1} holds.
\end{proof}

\begin{proof}[Proof of Proposition~\ref{prop:ft_bound}]
	Let $t,\alpha >0$ and let $f_{t,\alpha}$ be defined as in \eqref{eq:ft_def_inv}.
	The derivative of $f_{t,\alpha}$ is given by
	\[
	f_{t,\alpha}'(z) = 
	\left( \frac{t}{z (z+1)^{\alpha+1} }
	-\frac{\alpha z}{(z+1)^{\alpha + 2}} + \frac{1}{(z+1)^{\alpha + 2}} 
	\right) e^{-t/z}.
	\]
	Therefore,
	\begin{align*}
		&\sup_{\eta \in \mathbb{R}} |f_{t,\alpha}'(\xi+i \eta)| \\
		&\quad \leq 
		t\sup_{\eta \in \mathbb{R}} \frac{e^{-t\xi/(\xi^2+\eta^2)}  }{\big((\xi+1)^2+\eta^2\big)^{(\alpha+1)/2} \sqrt{\xi^2+\eta^2}} + 
		\alpha \sup_{\eta \in \mathbb{R}} \frac{ e^{-t\xi/(\xi^2+\eta^2)} \sqrt{\xi^2+\eta^2}  }{\big((\xi+1)^2+\eta^2\big)^{\alpha/2 + 1}} \\
		&\qquad + 
		\sup_{\eta \in \mathbb{R}} \frac{e^{-t\xi/(\xi^2+\eta^2)}}{\big((\xi+1)^2+\eta^2\big)^{\alpha/2 + 1}}
		\notag \\
		&\quad \eqqcolon \phi_{t,1}(\xi) + \phi_{t,2}(\xi) + \phi_{t,3}(\xi) 
	\end{align*}
	for all $\xi > 0$.
	To the term $\phi_{t,1}$, we apply
	Lemma~\ref{lem:g_inv} with $p =\alpha/2$ and $s = \eta^2$.
	Then we obtain
	\[
	\sup_{t > 0}\int_0^{\infty} \phi_{t,1}(\xi) d\xi < \infty.
	\]
	The term $\phi_{t,2}$ satisfies
	\[
	\phi_{t,2}(\xi)
	\leq \alpha 
	\sup_{\eta \in \mathbb{R}}
	\frac{\sqrt{(\xi+1)^2+\eta^2}}{\big((\xi+1)^2+\eta^2\big)^{\alpha/2+1}} 
	= \frac{\alpha}{(\xi+1)^{\alpha+1}}
	\]
	for all $\xi>0$.
	We also have that
	\[
	\phi_{t,3}(\xi) \leq \frac{1}{(\xi+1)^{\alpha+2}}
	\]
	for all $\xi>0$. Therefore,
	\[
	\sup_{t > 0}\int_0^{\infty} \phi_{t,2}(\xi) + \phi_{t,3}(\xi) d\xi < \infty.
	\]
	Thus, the desired estimate \eqref{eq:f_t_norm_bound} is obtained.
\end{proof}

We are now ready to prove Theorem~\ref{thm:Inv_bound}.
\begin{proof}[Proof of Theorem~\ref{thm:Inv_bound}]
	Let $t,\alpha >0$, and let the function $f_{t,\alpha}$ 
	be defined as in \eqref{eq:ft_def_inv}.
	As seen in Example~\ref{ex:LM_calculus}, we have
	\[
	f_{t,\alpha}(A) = A(A+I)^{-1} e^{-tA^{-1}} (A+I)^{-\alpha}.
	\]
	Hence
	\begin{equation}
		\label{eq:inv_gen_estimate1}
		\|e^{-tA^{-1}} (A+I)^{-\alpha}\| \leq \|I+A^{-1}\| \, \|f_{t,\alpha}(A)\|.
	\end{equation}
	By the estimate \eqref{eq:fA_bound_Bnorm},
	\begin{equation}
		\label{eq:inv_gen_estimate2}
		\|f_{t,\alpha}(A)\| \leq  2 K^2  \|f_{t,\alpha}\|_{\mathcal{B}_0},
	\end{equation}
	where $K \coloneqq \sup_{t \geq 0} \|e^{-tA}\|$.
	Combining the estimates \eqref{eq:inv_gen_estimate1} and \eqref{eq:inv_gen_estimate2}
	with Proposition~\ref{prop:ft_bound},
	we obtain the desired conclusion \eqref{eq:inv_bounded}.
\end{proof}

\subsection{Polynomially stable semigroups}
Here we consider a polynomially stable $C_0$-semigroup  $(e^{-tA})_{t \geq 0}$ 
on a Hilbert space. 
The next result shows how a decay estimate for $(e^{-tA})_{t \geq 0}$ 
can be transferred to that for $(e^{-tA^{-1}})_{t \geq 0}$.
As in Theorem~\ref{thm:CT_bound_poly}, 
we use the notation $\lfloor \xi \rfloor \coloneqq
\max\{k \in \mathbb{Z}: k \leq \xi \}$ for $\xi \in \mathbb{R}$.
\begin{theorem}
	\label{thm:Inv_bound_poly}
	Let $-A$ be the generator of a polynomially stable $C_0$-semigroup 
	with parameter $\beta >0$
	on a Hilbert space $H$. Then,
	for all $\alpha >0$, 
	\begin{equation}
		\label{eq:Inv_bound}
		\|e^{-tA^{-1}} A^{-\alpha}\| = O\left(
		\frac{(\log t)^{k+1} }{t^{\alpha/(2+\beta)}}
		\right)\quad (t \to \infty),\quad \text{where~}
		k \coloneqq \left\lfloor
		\frac{2\alpha}{2+\beta}
		\right \rfloor.
	\end{equation}
\end{theorem}

To prove Theorem~\ref{thm:Inv_bound_poly}, we need the following estimate 
for $\|f_{t,2p}\|_{\mathcal{B}_0, q }$ with $0 < q < 1/2$, where 
$f_{t,2p}$ is defined by
\eqref{eq:ft_def_inv} with $\alpha = 2p$.
\begin{proposition}
	\label{prop:f_t_norm}
	Let $p >0$ and $ q  \in (0,1/2)$.  Then 
	the function $f_{t,2p}$ defined by \eqref{eq:ft_def_inv} with $\alpha = 2p$
	satisfies
	\begin{equation}
		\label{eq:f_t_bound}
		\|f_{t,2p}\|_{\mathcal{B}_0, q } =	\begin{cases}
			O\left(
			\dfrac{1}{t^{p}}
			\right), & p < q , \vspace{3pt}\\
			O\left(
			\dfrac{\log t}{t^{p}}
			\right), & p = q , \vspace{3pt}\qquad (t \to \infty).\\
			O\left(
			\dfrac{1}{t^{ q }}
			\right), & p >  q ,
		\end{cases}
	\end{equation}
\end{proposition}
\begin{proof}
	Step~1:
	Fix $p >0$ and $ q  \in (0,1/2)$.
	The derivative $f_{t,2p}'$ is given by
	\[
	f_{t,2p}'(z) = 
	\left(
	\frac{t}{z(z+1)^{2p+1}} 
	-\frac{2p z}{(z+1)^{2p+2}}  +
	\frac{1}{(z+1)^{2p+2}} 
	\right) e^{-t/z}.
	\]
	Since
	\[
	\left|
	\frac{1}{(z+1)^{2p+2}} 
	\right| 
	\leq
	\left|
	\frac{1}{z(z+1)^{2p+1}} 
	\right|
	\]
	for all $z \in \mathbb{C}_+$,
	we have
	\begin{equation}
		\label{eq:f_gh_bound_Inv}
		\sup_{\eta \in \mathbb{R}}
		|f_{t,2p}' (\xi +i\eta)| \leq (t+1)
		\sup_{s \geq 0}g_{t,p}(\xi ,s) +
		2p
		\sup_{s \geq 0} h_{t,p}(\xi ,s),
	\end{equation}
	where 
	$g_{t,p}$ is defined by \eqref{eq:g_ta_inv}
	and
	\begin{align}
		\label{eq:h_tp_inv}
		h_{t,p}(\xi ,s) &\coloneqq
		\frac{ e^{-t\xi /(\xi ^2+s)}\sqrt{\xi ^2+s} }{
			\big((\xi +1)^2+s\big)^{p+1} 
		} .
	\end{align}
	for $\xi>0$ and $s \geq 0$.
	We will prove in Steps~2 and 3 below that
	\begin{align}
		\int_0^{\infty} \psi_q(\xi ) \sup_{s \geq 0}g_{t,p}(\xi ,s) d\xi  &=
		\begin{cases}
			O\left(
			\dfrac{1}{t^{p+1}}
			\right), & p < q , \vspace{3pt}\\
			O\left(
			\dfrac{\log t}{t^{p+1}}
			\right), & p = q , \qquad (t \to \infty) \vspace{3pt}\\
			O\left(
			\dfrac{1}{t^{ q +1}}
			\right), & p >  q ,
		\end{cases} \label{eq:gt_bound} 
	\end{align}
	and
	\begin{align}
		\int_0^{\infty} \psi_q(\xi )  \sup_{s \geq 0} h_{t,p}(\xi ,s) d\xi  &=
		\begin{cases}
			O\left(
			\dfrac{1}{t^{2p}}
			\right), &  p< \dfrac{1}{2}, \vspace{3pt}  \\
			O\left(
			\dfrac{\log t}{t}
			\right), &  p= \dfrac{1}{2}, \qquad (t \to \infty).\vspace{3pt}  \\
			O\left(
			\dfrac{1}{t}
			\right), &  p> \dfrac{1}{2},
		\end{cases}
		\label{eq:ht_bound}
	\end{align}
	The desired conclusion \eqref{eq:f_t_bound} follows immediately from
	the estimates \eqref{eq:f_gh_bound_Inv}--\eqref{eq:ht_bound}.
	
	Step~2:
	In this step, we will prove the estimate \eqref{eq:gt_bound} for $g_{t,p}$.
	From the estimates \eqref{eq:v_zeta_bound} and \eqref{eq:sup_v_zeta_bound} in the proof of Lemma~\ref{lem:g_inv},
	we obtain
	\begin{equation}
		\label{eq:g_tp_bound}
		\sup_{s \geq 0}g_{t,p}(\xi ,s) \leq \frac{1}{(c_1t\xi + 1)^{p+1/2} \sqrt{c_1t\xi }}
	\end{equation}
	for all 
	$t,\xi>0$ and 
	some constant $c_1>0$ depending only on $p$.

	Suppose that $t > 1$.
	Since
	the function $\psi_q$, defined as in \eqref{eq:psi_def}, satisfies
	$\psi_q(\xi) = \xi^q$ for all $\xi \in (0,1)$, 
	the estimate \eqref{eq:g_tp_bound} gives
	\begin{align*}
		\int_0^{1} \psi_{ q }(\xi ) \sup_{s \geq 0}g_{t,p}(\xi ,s) d\xi   
		&\leq \int^1_0 \frac{\xi ^{ q }}{(c_1t\xi  + 1)^{p+1/2} \sqrt{c_1t\xi }} d\xi  \\
		&=
		\frac{1}{(c_1t)^{ q +1}}
		\int_0^{c_1t} \frac{\xi ^{ q }}{(\xi  + 1)^{p+1/2} \sqrt{\xi }} d\xi  \\
		&\leq
		\frac{1}{(c_1t)^{ q +1}}
		\left(
		\int_0^{c_1} \frac{1}{\xi ^{1/2- q }} d\xi  + 
		\int_{c_1}^{c_1t} \frac{1}{\xi ^{p- q  + 1}} d\xi  
		\right).
	\end{align*}
	Hence
	\begin{equation}
		\label{eq:g_t_below}
		\int_0^{1} \psi_{ q }(\xi ) \sup_{s \geq 0}g_{t,p}(\xi ,s) d\xi   
		=
		\begin{cases}
			O\left(
			\dfrac{1}{t^{p+1}}
			\right), & p <  q , \vspace{3pt} \\
			O\left(
			\dfrac{\log t}{t^{p+1}}
			\right), & p =  q , \vspace{3pt} \\
			O\left(
			\dfrac{1}{t^{ q +1}}
			\right), & p >  q 
		\end{cases}
	\end{equation}
	as $t \to \infty$.
	Recalling that $\psi_q(\xi) = 1$ for all $\xi \geq  1$, 
	we also have
	\begin{align*}
		\int_1^{\infty}\psi_{ q }(\xi ) \sup_{s \geq 0}g_{t,p}(\xi ,s) d\xi   
		\leq
		\frac{1}{c_1t}
		\int_{c_1t}^{\infty} \frac{1}{(\xi +1)^{p+1/2} \sqrt{\xi }} d\xi  
		\leq 
		\frac{1}{c_1t}
		\int_{c_1t}^{\infty} \frac{1}{\xi ^{p+1}} d\xi ,
	\end{align*}
	which yields
	\begin{equation}
		\label{eq:g_t_above}
		\int_1^{\infty} \psi_{ q }(\xi ) \sup_{s \geq 0}g_{t,p}(\xi ,s) d\xi   =
		O\left(
		\frac{1}{t^{p+1}}
		\right)
	\end{equation}
	as $t \to \infty$.
	From  \eqref{eq:g_t_below} and \eqref{eq:g_t_above},
	we obtain the estimate \eqref{eq:gt_bound}  for $g_{t,p}$.
	
	Step~3:
	It remains to prove the estimate \eqref{eq:ht_bound} for $h_{t,p}$.
	Define $\zeta \coloneqq t\xi $ and
	\[
	H_{\zeta} (\tau) \coloneqq \frac{e^{-\zeta/\tau }}{(\tau+1)^{p+1/2}  },
	\quad \tau >0.
	\]
	Then $H_{\zeta}$ has similar properties to $G_{\zeta}$ used in the proof of Lemma~\ref{lem:g_inv}. In fact,
	\begin{equation}
		\label{eq:h_tp_estimate1}
		h_{t,p}(\xi ,s) \leq 
		H_{\zeta}(\xi ^2+s)
	\end{equation}
	for all $t, \xi>0$ and $s \geq 0$. There exists a unique positive solution 
	$\tau_2 = \tau_2(\zeta)$ of $H_{\zeta}'(\tau) = 0$, and
	\[
	\sup_{\tau \geq 0 } H_{\zeta}(\tau) =
	H_{\zeta}(\tau_2).
	\]
	Moreover,
	$
	\tau_2 \geq c_2 \zeta
	$
	for some constant $c_2>0$ depending only on $p$.
	Therefore,
	\begin{equation}
		\label{eq:h_tp_estimate2}
		\sup_{s \geq 0}H_{\zeta} (\xi ^2+s) \leq
		H_{\zeta}(\tau_2) \leq \frac{1}{(c_2\zeta + 1)^{p+1/2}}
	\end{equation}
	for all $t,\xi >0$.
	
	By definition, $\psi_{ q } (\xi ) \leq 1$ for all $\xi  >0$. Combining this with the estimates 
	\eqref{eq:h_tp_estimate1} and \eqref{eq:h_tp_estimate2}, we obtain
	\begin{align*}
		\int_0^{t} \psi_{ q }(\xi ) \sup_{s \geq 0} h_{t,p}(\xi ,s) d\xi  
		&\leq 
		\int_0^{t} \frac{1}{(c_2t\xi + 1)^{p+1/2}} d\xi  = \frac{1}{c_2 t }
		\int_0^{c_2t^2} \frac{1}{(\xi  + 1)^{p+1/2} } d\xi
	\end{align*}
	for all $t >0$.
	Hence
	\begin{equation}
		\label{eq:h_t_below}
		\int_0^{t}\psi_{ q }(\xi ) \sup_{s \geq 0} h_{t,p}(\xi ,s) d\xi   = 
		\begin{cases}
			O\left(
			\dfrac{1}{t^{2p}}
			\right), &  p< \dfrac{1}{2}, \vspace{3pt}  \\
			O\left(
			\dfrac{\log t}{t}
			\right), &  p= \dfrac{1}{2}, \vspace{3pt}  \\
			O\left(
			\dfrac{1}{t}
			\right), &  p> \dfrac{1}{2}
		\end{cases}
	\end{equation}
	as $t\to\infty$.
	On the other hand, we have from the definition \eqref{eq:h_tp_inv} of $h_{t,p}$ that
	\[
	h_{t,p}(\xi ,s) \leq \frac{1}{\xi ^{2p+1}}
	\]
	for all $t, \xi>0$ and $s \geq 0$. This gives
	\begin{equation}
		\label{eq:h_t_above}
		\int_t^{\infty} 
		\psi_q(\xi)
		\sup_{s \geq 0} h_{t,p}(\xi ,s) d\xi \leq 
		\int_t^{\infty}
		\frac{1}{\xi^{2p+1}} d\xi
		=
		O\left(
		\frac{1}{t^{2p}}
		\right)
	\end{equation}
	as $t\to\infty$.
	From  \eqref{eq:h_t_below} and \eqref{eq:h_t_above},
	we obtain the estimate \eqref{eq:ht_bound}  for $h_{t,p}$.
\end{proof}

Now we are in a position to give the proof of Theorem~\ref{thm:Inv_bound_poly}.
\begin{proof}[Proof of Theorem~\ref{thm:Inv_bound_poly}]
	Let $p \in (0,1/2)$ and define $f_{t,2p} \in \mathcal{B}$ by 
	\eqref{eq:ft_def_inv} with $\alpha = 2p$ for $t >0$.
	Using 
	Example~\ref{ex:bound_poly} in the case $\alpha = 2p$, $c= 1$, $q = p$, and
	\[
	f(z) = \frac{z}{z+1} e^{-t/z},
	\]
	we see that 
	there exists a constant $M>0$ such that 
	\[
	\|e^{-tA^{-1}} A^{-(2+\beta)p} \| \leq M \|f_{t,2p} \|_{\mathcal{B}_0,p}
	\]
	for all $t >0$.
	The estimate \eqref{eq:f_t_bound}
	with $p =  q$ yields
	\[
	\|e^{-tA^{-1}} A^{-(2+\beta)p} \| = O\left(
	\frac{\log t}{t^p}
	\right)
	\]
	as $t \to \infty$.
	Therefore, the desired estimate \eqref{eq:Inv_bound} holds in the case 
	$k = \lfloor  2\alpha/(2+\beta) \rfloor=0$.
	The case $k \geq 1$ follows as in the proof of 
	Theorem~\ref{thm:CT_bound_poly}.
\end{proof}

\begin{remark}
	The estimate \eqref{eq:f_t_bound}
	with $p =  q $ for $\|f_{t,2p}\|_{\mathcal{B}_0, q }$
	has been used in the proof of Theorem~\ref{thm:Inv_bound_poly}. 
	This is because 
	the case $p \not=  q $ yields
	a less sharp estimate, which can be seen from the 
	same argument as in Remark~\ref{rem:p_gamma}.
\end{remark}

\noindent
\textbf{Acknowledgements}
This work was supported by JSPS KAKENHI Grant Number JP20K14362.
The author would like to thank Yuri Tomilov for providing useful comments on 
Remarks~\ref{rem:optimality1} and \ref{rem:optimality2}
and for sharing the manuscript of \cite{Gomilko2023}.

\end{document}